\documentclass[reqno,oneside,12pt]{amsart}
\pdfoutput=1

\usepackage{
  amsmath,
  amssymb,
  amsthm,
  geometry,
  paralist,
  thmtools,
  tikz,
  tikz-cd,
  verbatim,
  relsize, 
  bbm, 
  dsfont,
  mathtools
}

\usepackage[T1]{fontenc}

\usepackage{caption}
\usepackage{subcaption}
\usetikzlibrary{angles}

\renewcommand\labelenumi{(\roman{enumi})}
\renewcommand\theenumi\labelenumi

\usepackage{tcolorbox}

\RequirePackage{ifpdf}
\ifpdf
  \IfFileExists{pdfsync.sty}{\RequirePackage{pdfsync}}{}
  \RequirePackage[pdftex,
   colorlinks = true,
   urlcolor = myblue, 
   citecolor = mygreen, 
   linkcolor = myred, 
   bookmarksopen=true,
   unicode, psdextra]{hyperref}
\else
  \RequirePackage[hypertex, unicode, psdextra]{hyperref}
\fi
\hypersetup{
    colorlinks=true,
    linkcolor=blue,
    filecolor=magenta,      
    urlcolor=cyan,
    citecolor = blue
} 

\RequirePackage{ae, aecompl, aeguill} 

\usepackage{quiver}

\usepackage{cleveref} 

\usepackage{adjustbox}

\tikzcdset{
	arrows = crossing over
}

\usepackage[maxbibnames=99,
			maxalphanames=10, 
			style=alphabetic]{biblatex}
\usepackage{xurl}

\usepackage{dynkin-diagrams}

\usepackage[boxsize=.5em]{ytableau}

\usetikzlibrary{arrows.meta} 

\usepackage[textsize=scriptsize]{todonotes}
\presetkeys{todonotes}{inline}{}

\makeatletter
\providecommand\@dotsep{5}
\makeatother

\newcommand{\bbB}{\mathbb B}
\newcommand{\bbC}{\mathbb C}
\newcommand{\bbH}{\mathbb H}

\newcommand{\bbR}{\mathbb R}

\newcommand{\bbW}{\mathbb W}
\newcommand{\bbZ}{\mathbb Z}

\newcommand{\cA}{\mathcal A}
\newcommand{\cB}{\mathcal B}

\newcommand{\cD}{\mathcal D}
\newcommand{\cF}{\mathcal F}

\newcommand{\cH}{\mathcal H}

\newcommand{\cK}{\mathcal K}

\newcommand{\cP}{\mathcal P}
\newcommand{\cQ}{\mathcal Q}

\newcommand{\cZ}{\mathcal Z}

\newcommand{\frW}{\mathfrak{W}}

\newcommand{\<}{\langle}
\renewcommand{\>}{\rangle}
\newcommand{\wt}{\widetilde}

\DeclareMathOperator{\id}{id}
\DeclareMathOperator{\Z}{\mathbb Z}
\DeclareMathOperator{\R}{\mathbb R}

\DeclareMathOperator{\C}{\mathbb C}

\DeclareMathOperator{\ra}{\rightarrow}
\newcommand{\xra}{\xrightarrow}

\DeclareMathOperator{\cone}{cone}
\DeclareMathOperator{\Hom}{Hom}

\DeclareMathOperator{\Stab}{Stab}

\DeclareMathOperator{\Kom}{Kom}

\DeclareMathOperator{\gr}{gr}

\declaretheorem[numberwithin=section]{theorem}
\declaretheorem[sibling=theorem]{lemma}
\declaretheorem[sibling=theorem]{corollary}

\declaretheorem[sibling=theorem]{proposition}
\declaretheorem[sibling=theorem, style=remark]{remark}
\declaretheorem[sibling=theorem, style=definition]{definition}
\declaretheorem[sibling=theorem, style=definition]{example}

\declaretheorem[sibling=theorem, style=definition]{exercise}

\newtheorem*{theorem*}{Theorem}

\crefname{lemma}{Lemma}{Lemma}
  \crefname{corollary}{Corollary}{Corollary}
  \crefname{theorem}{Theorem}{Theorem}
  \crefname{definition}{Definition}{Definition}
   \crefname{proposition}{Proposition}{Proposition}
 \crefname{section}{Section}{Section} 
   \crefname{construction}{Construction}{Construction}
   \crefname{generalization}{Generalization}{Generalization}
  \crefname{construction}{Construction}{Construction}
  \crefname{notation}{Notation}{Notation}
   \crefname{example}{Example}{Example}
  \crefname{remark}{Remark}{Remark}
  \crefname{fact}{Fact}{Fact}
  \crefname{conjecture}{Conjecture}{Conjecture}
  \crefname{motivation}{Motivation}{Motivation}  
  \crefname{figure}{Figure}{Figure}  
  \crefname{assumption}{Assumption}{Assumption}
  \crefname{exercise}{Exercise}{Exercise}

\renewcommand{\comment}[1]{}

\begin{document}

\title[Intro to stability conditions and $K(\pi,1)$]{Introduction to stability conditions and its relation to the $K(\pi,1)$ conjecture for Artin groups}

\author[]{Edmund Heng}
\address{Institut des Hautes Etudes Scientifiques (IHES). Le Bois-Marie, 35, route de Chartres, 91440 Bures-sur-Yvette (France)}
\email{heng@ihes.fr}

\begin{abstract}
We give a brief introduction to the relationship between stability conditions and the $K(\pi,1)$ conjecture for Artin groups.
These notes have been written as pre-reading for the MFO mini-workshop 2405a \emph{Artin groups meet triangulated categories}.
\end{abstract}

\maketitle

\tableofcontents

\section{Introduction}
The aim of these notes is to give a brief introduction to Bridgeland stability conditions and to discuss their relation to the $K(\pi,1)$ conjecture for Artin groups.
The main take away is as follows: for each simply-laced\footnote{The non-simply-laced case will be briefly discussed in \cref{rem:nonsimplylaced}.} Coxeter graph $\Gamma$, the space of stability conditions $\Stab(\cK_\Gamma)$ (of some triangulated category $\cK_\Gamma$ associated to $\Gamma$) is a covering space of the hyperplane complement associated to $\Gamma$.
The general conjecture is that spaces of stability conditions are always contractible, which implies the $K(\pi,1)$ conjecture. 
In fact, for $\Gamma$ = ADE, the contractibility of $\Stab(\cK_\Gamma)$ can be proven independently of Deligne's result. 

These notes are organised as follows.
In \cref{sec:admissiblewalks}, we associate a category of admissible walks $\wt{\Gamma}$ to each simply-laced Coxeter graph $\Gamma$. 
In \cref{sec:homotopycat}, we introduce the (formal) additive completion of $\widetilde{\Gamma}$ into an additive category $\cA_{\Gamma}$, whose homotopy category $\cK_{\Gamma}\coloneqq \Kom^b(\cA_{\Gamma})$ will be the triangulated category of interest.
In \cref{sec:artinact}, we discuss the action of the Artin group $\bbB(\Gamma)$ on the triangulated category $\cK_\Gamma$.
In \cref{sec:triangulated} and \cref{sec:heartlinear}, we discuss some important structures associated to $\cK_\Gamma$ (triangulated structure and hearts respectively). 
We then introduce stability conditions and the complex manifold structure on the space of stability condition $\Stab(\cD)$ in \cref{sec:stab}.
Finally in \cref{sec:coveringstab}, we see that $\Stab(\cK_{\Gamma})$ is a covering space for the (quotient of) hyperplane complement associated to the Coxeter group $\bbW(\Gamma)$ for $\Gamma = $ ADE, with $\bbB(\Gamma)$ acting via deck transformations.
The fact that $\Stab(\cK_{\Gamma})$ is contractible then gives a new proof of the $K(\pi,1)$ conjecture.
We briefly discuss the cases of non-finite types and non-simply-laced types $\Gamma$ at the end of the section.

\subsection*{Acknowledgement} 
These notes were written for the Oberwolfach mini-workshop 2405a: \emph{Artin groups meet triangulated categories}.
A companion set of notes can be found in \cite{Boyd_MFOnotes}.
The corresponding Oberwolfach Report is \cite{MFOreport}.

The author would like to thank the workshop co-organisers, Rachael Boyd and Viktoriya Ozornova, for their wonderful support throughout the organisation and preparation of this workshop.
The author would also like to thank the participants for their feedback on the draft(s) of these notes.
Special thanks to Jon McCammond and Tony Licata for running the lecture series of this workshop.
The workshop organisers would like to thank MFO for the opportunity to organise a mini-workshop on this topic.
The MFO and the workshop organisers would like to thank the National Science Foundation for supporting the participation of junior researchers in the workshop by the grant DMS-2230648, ``US Junior Oberwolfach Fellows''.

\section{The category of admissible walks on graphs}\label{sec:admissiblewalks}
Let $\Gamma$ be a simply-laced Coxeter graph\footnote{We use the Coxeter graph convention where two vertices $i,j$ not connected by an edge means their Coxeter generators commute in the Coxeter group.}; this means that all edges in our graph are unlabelled and between any two vertices there is at most one edge connecting them.
We denote its set of vertices by $\Gamma_0$ and its set of edges $\Gamma_1$.
Since $\Gamma$ is simply-laced, if $i,j \in \Gamma_0$ are connected by an edge, we shall use $(i,j)=(j,i)$ to denote the edge in $\Gamma_1$.

A \emph{walk} on a graph is a finite sequence of edges which joins a sequence of vertices (we allow the same edge to be used more than once).
Given an edge $(i,j) \in \Gamma_1$, we shall use $(i|j)$ to denote the (length $1$) walk from from vertex $i$ to vertex $j$, and conversely $(j|i)$ will denote the opposite walk from vertex $j$ to vertex $i$.
We warn the reader that while both walks use the same edge $(i,j)=(j,i)$, they are different as walks, i.e.\ $(i|j) \neq (j|i)$.
We will use the function composition convention for walks -- the sequence of edges is read from right to left.
For example, the walk from $i$ to $j$ and then $j$ to $k$ is denoted by $(j|k)(i|j)$ (this convention will make our lives easier later on).
By definition, we shall also allow the \emph{constant walk} on a vertex $i$: it is denoted by $e_i$ and it has length 0 by definition.
Nonetheless, we shall not differentiate between walks that have extra constant walks sandwiched between; namely for any walk $p$ starting at $i$ and ending at $j$, we have
\[
pe_i = p = e_jp \text{ as walks}.
\]
In particular, $e_i e_i = e_i$.

By definition, a walk $p$ on $\Gamma$ is said to be \emph{admissible} if it satisfies (both of) the following conditions:
\begin{itemize}
\item the length of $p$ $\leq 2$; and
\item if $p$ has length 2, then it must start and end at the same vertex.
\end{itemize}
Note that this includes the constant walks $e_i$'s.
As such, starting from a vertex $i$, there are only three types of admissible walks:
\begin{enumerate}
\item[(length 0)] the constant walk $e_i$;
\item[(length 1)] for each neighbouring vertex $j$ (i.e.\ $(i,j) \in \Gamma_1$), we have $(i|j)$; and 
\item[(length 2)] for each neighbouring vertex $j$, we have $(j|i)(i|j)$ (a closed walk at $i$).
\end{enumerate}
On the set of admissible walks, we define the following equivalence relation:
\[
\text{for each } i \in \Gamma_0, \quad (j|i)(i|j) \sim (k|i)(i|k).
\]
In other words, for each vertex $i$, all the closed walks (of length 2) at $i$ fall into the same equivalence class, which we shall denote by $X_i$.
The composition of two equivalence classes of admissible walks can still be defined: we compose them as walks by picking \emph{any} representative -- the caveat is that if the composition is no longer admissible, we send it to the empty set.
For example, the following compositions are all sent to the empty set:
\begin{enumerate}
    \item $(j|k)(i|j)$ for $i \neq k$ (length 2 with different starting and ending vertex);
    \item $(i|j)X_i$ and $X_i(j|i)$ (length $> 2$).
\end{enumerate}
Abusing notation, the equivalence classes of admissible walks will still be called admissible walks.

\begin{definition}
We associated to $\Gamma$ a category $\wt{\Gamma}$ as follows.
\begin{itemize}
\item The objects are $P_i\<m\>$ for each $i \in \Gamma_0$ and $m \in \Z$. When $m=0$, we will simply write $P_i \coloneqq P_i\<0\>$.
\item For each pair of vertices $i, j \in \Gamma_0$ (not necessarily distinct), the space of morphisms
$\hom(P_i\<m\>, P_j\<n\>)$ from $P_i\<m\>$ to $P_j\<n\>$ is given by the formal $\bbC$-linear span of admissible walks from $i$ to $j$ of length $m-n$ (zero if empty):
\begin{equation} \label{eq:admissiblepaths}
\hom(P_i\<m\>, P_j\<n\>) \coloneqq \begin{cases}
	\bbC \cdot e_i, &\text{ if } i=j, m=n; \\
	\bbC \cdot X_i, &\text{ if } i=j, m-n=2; \\
	\bbC \cdot (i|j), &\text{ if } (i,j)=(j,i) \in \Gamma_1, m-n=1; \\
	\bbC \cdot \emptyset \coloneqq 0, &\text{ otherwise}.
	\end{cases}
\end{equation}
\end{itemize}
\end{definition}
The composition of morphisms is as in composition of admissible walks (the scalars simply multiply); if the composition is no longer admissible, it is (by definition) zero.

\section{The homotopy category of complexes \texorpdfstring{$\cK_{\Gamma}$}{K}} \label{sec:homotopycat}
Recall that a (cochain) complex (in some additive category $\cA$) is a sequence of morphisms
\[
\cdots \xra{d_{-2}} C^{-1} \xra{d_{-1}} C^0 \xra{d_0} C^1 \xra{d_1} C^2 \xra{d_2} \cdots
\]
such that $d_{i+1} d_i = 0$ ($d^2 = 0$).
For our purposes, all complexes will be \emph{bounded}, so that $C^i = 0 = C^{-i}$ for large $i \gg 0$.
A morphism of complexes from $(C^i, d_i)$ to $(D^i, d'_i)$ is a collection of morphisms $f_i\colon C^i \ra D^i$ for each $i \in \Z$ such that all squares in the following diagram commute:
\begin{equation} \label{eq:morphcomplex}
\begin{tikzcd}
\cdots \ar[r, "d_{-2}"] & C^{-1} \ar[r, "d_{-1}"] \ar[d, "f_{-1}"] & C^0 \ar[r, "d_0"] \ar[d, "f_0"] & C^1 \ar[r, "d_1"]  \ar[d, "f_1"] & C^2 \ar[r, "d_2"]  \ar[d, "f_2"] & \cdots \\
\cdots \ar[r, "d'_{-2}"] & D^{-1} \ar[r, "d'_{-1}"] & D^0 \ar[r, "d'_0"] & D^1 \ar[r, "d'_1"] & D^2 \ar[r, "d'_2"] & \cdots 
\end{tikzcd}
\end{equation}

A \emph{homotopy} of complexes from $(C^i, d_i)$ to $(D^i, d'_i)$ is a collection of (degree $-1$) morphisms $h_i\colon C^i \ra D^{i-1}$ (warning: no commutativity requirement!):
\[
\begin{tikzcd}
\cdots \ar[r, "d_{-2}"] & C^{-1} \ar[r, "d_{-1}"] & C^0 \ar[r, "d_0"] \ar[dl, "h_0", swap] & C^1 \ar[r, "d_1"]  \ar[dl, "h_1", swap] & C^2 \ar[r, "d_2"]  \ar[dl, "h_2", swap] & \cdots \\
\cdots \ar[r, "d'_{-2}"] & D^{-1} \ar[r, "d'_{-1}"] & D^0 \ar[r, "d'_0"] & D^1 \ar[r, "d'_1"] & D^2 \ar[r, "d'_2"] & \cdots 
\end{tikzcd}
\]
Note that $d'_{i-1} h_i + h_{i+1} d_i$ is a morphism from $C^i$ to $D^i$.
Two morphisms of complexes $(f_i), (f'_i)\colon (C^i, d_i) \ra (D^i, d'_i)$ are \emph{homotopic} if there exists a homotopy $(h_i)$ such that 
\[
f_i + d'_{i-1} h_i + h_{i+1} d_i = f'_i
\]
for all $i \in \bbZ$.
In particular, a morphism of complexes $(f_i)\colon (C^i, d_i) \ra (D^i, d'_i)$ is a \emph{homotopy equivalence} if there exists a morphism of complexes $(g_i)\colon (D^i, d'_i) \ra (C^i, d_i)$ such that
\[
g_i f_i + d_{i-1} h_i + h_{i+1} d_i = \id_{C^i}, \quad f_i g_i + d'_{i-1} h'_i + h'_{i+1} d'_i = \id_{D^i}
\]
for some homotopies $(h_i)$ and $(h'_i)$.

The \emph{cone} of a morphism $f\coloneqq (f_i)\colon (C^i, d_i) \ra (D^i, d'_i)$ is a complex defined by
\begin{equation} \label{eq:cone}
\cone(f)\coloneqq
\begin{tikzcd}
    \cdots                        \ar[r, "-d_{-1}"] \ar[rd, "f_{-1}"description] & 
	C^0 \ar[d, phantom, "\oplus"] \ar[r, "-d_0"] \ar[rd, "f_0"description] & 
	C^1 \ar[d, phantom, "\oplus"] \ar[r, "-d_1"] \ar[rd, "f_1"description] & 
	C^2 \ar[d, phantom, "\oplus"] \ar[r, "-d_2"] \ar[rd, "f_2"description] & 
	C^3 \ar[d, phantom, "\oplus"] \ar[r, "-d_3"] \ar[rd, "f_3"description] & 
	\cdots \\
\cdots \ar[r, "d'_{-2}"] & D^{-1} \ar[r, "d'_{-1}"] & D^0 \ar[r, "d'_0"] & D^1 \ar[r, "d'_1"] & D^2 \ar[r, "d'_2"] & \cdots 
\end{tikzcd},
\end{equation}
where $C^1 \oplus D^0$ sits in cohomological degree zero.
Notice the negative signs on the $d_i$'s; these ensure $\cone(f)$ is indeed a complex.

We would like to consider (the homotopy category of) complexes of $\wt{\Gamma}$, but there is no zero object in $\wt{\Gamma}$. 
Moreover, taking `cone' will not make sense: we can't take direct sums in $\wt{\Gamma}$.
Luckily, there is an easy fix by formally adding direct sums of objects as follows.
\begin{definition}
We define the category $\cA_{\Gamma}$ to be the additive
completion of $\wt{\Gamma}$.
The objects are formal finite direct sums $\oplus_i x_i$ of objects $x_i$ in $\wt{\Gamma}$ (the $x_i$'s need not be distinct); the empty direct sum is by definition the zero object $0$.

The morphism spaces are defined as follows:
\[
\Hom_{\cA_{\Gamma}}(\oplus_{i=1}^a x_i, \oplus_{j=1}^b y_j) \coloneqq \bigoplus_{i,j} \hom(x_i, y_j),
\]
where each $f\coloneqq(f_{j,i})\colon \oplus_{i=1}^a x_i \ra \oplus_{j=1}^b y_j$ is written as an $b\times a$ matrix whose columns are indexed by $x_i$ and rows are indexed by $y_j$, with entries $f_{j,i} \in \hom(x_i, y_j)$.
The composition and addition of morphisms are done following the rule of matrices.
\end{definition}
\begin{remark}
We warn the reader that $\cA_\Gamma$ is only an additive ($\bbC$-linear) category and it is \emph{not} an abelian category.
\end{remark}
\begin{remark}
The additive category $\cA_\Gamma$ can be realised as the category of graded, (finitely generated) projective modules over the zigzag algebra associated to $\Gamma$ (with the path length grading); see \cite{HueKho, LQ_ADE}.
\end{remark}
\begin{example}
Let $\Gamma = A_2 = \begin{tikzcd}[column sep=small] 1 \ar[r,no head] & 2 \end{tikzcd}$.
Consider the two morphisms in $\cA_{A_2}$:
\[
	{P_1\<3\>\oplus P_1\<1\>}
	    \xra{\begin{bsmallmatrix} X_1 & e_1 \\ 0 & (1|2) \end{bsmallmatrix}}
	{P_1\<1\>\oplus P_2}, \quad
	{P_1\<1\>\oplus P_2}
	    \xra{\begin{bsmallmatrix} -X_1 & (2|1) \\ (1|2) & e_2 \end{bsmallmatrix}}
	{P_1\<-1\>\oplus P_2}.
\]
Their composition is given by
\begin{align*}
\begin{bmatrix} -X_1 & (2|1) \\ (1|2) & e_2 \end{bmatrix} \begin{bmatrix} X_1 & e_1 \\ 0 & (1|2) \end{bmatrix}
    &= \begin{bmatrix} (-X_1)(X_1) & (-X_1)(e_1) + (2|1)(1|2) \\ (1|2)(X_1) & (1|2)(e_1) + (e_2)(1|2) \end{bmatrix} \\
    &= \begin{bmatrix} 0 & 0 \\ 0 & 2\cdot (1|2) \end{bmatrix}.
\end{align*}
\end{example}
\begin{remark}[Notation] \label{rem:complexnotation}
Sometimes it is easier to present the morphisms between the direct summands by explicit arrows between them, where it is implicit that no arrows means it is given by the zero morphism. 
For example, the morphism 
\[
    {P_1\<3\>\oplus P_1\<1\>}
	    \xra{\begin{bsmallmatrix} X_1 & e_1 \\ 0 & (1|2) \end{bsmallmatrix}}
    {P_1\<1\>\oplus P_2}    
\]
in the previous example can be written as:
\[\begin{tikzcd}
	{P_1\<3\>} & {P_1\<1\>} \\
	{P_1\<1\>} & {P_2}
	\arrow["{X_1}", from=1-1, to=1-2]
	\arrow["{(1|2)}", from=2-1, to=2-2]
	\arrow["{e_1}", from=2-1, to=1-2]
	\arrow["\oplus"{marking, allow upside down}, draw=none, from=1-1, to=2-1]
	\arrow["\oplus"{marking, allow upside down}, draw=none, from=1-2, to=2-2]
\end{tikzcd};\]
whereas the resulting composition given by the matrix $\begin{bsmallmatrix} 0 & 0 \\ 0 & 2\cdot(1|2) \end{bsmallmatrix}$ is written as:
\[\begin{tikzcd}
	{P_1\<3\>} & {P_1\<-1\>} \\
	{P_1\<1\>} & {P_2}
	\arrow["{2\cdot (1|2)}", from=2-1, to=2-2]
	\arrow["\oplus"{marking, allow upside down}, draw=none, from=1-1, to=2-1]
	\arrow["\oplus"{marking, allow upside down}, draw=none, from=1-2, to=2-2]
\end{tikzcd}.\]
\end{remark}
\begin{definition}
The \emph{homotopy category of complexes in $\cA_{\Gamma}$}, denoted by $\cK_{\Gamma}$, is the category whose objects are complexes in $\cA_{\Gamma}$ and morphisms are morphisms of complexes \emph{modulo homotopy}.
\end{definition}
Note that this means isomorphisms in $\cK_{\Gamma}$ are homotopy equivalences.
From here on, we will use both ``$X^\bullet$ is isomorphic to $Y^\bullet$ in $\cK_{\Gamma}$'' (or $X^\bullet \cong Y^\bullet \in \cK_{\Gamma}$ for short) and ``$X^\bullet$ is homotopy equivalent to $Y^\bullet$'' interchangeably.
\begin{exercise}
Let $\Gamma = A_2 = \begin{tikzcd}[column sep=small] 1 \ar[r,no head] & 2 \end{tikzcd}$.
Show that the morphism of complexes in $\cK_{A_2}$
\[\begin{tikzcd}
	0 & 0 & {P_1\<1\>} & 0 \\
	0 & {P_1\<1\>} & {P_2} & 0
	\arrow["{(1|2)}", from=2-2, to=2-3]
	\arrow[from=2-1, to=2-2]
	\arrow[from=2-3, to=2-4]
	\arrow["{(1|2)}", from=1-3, to=2-3]
	\arrow[from=1-2, to=1-3]
	\arrow[from=1-3, to=1-4]
	\arrow[from=1-2, to=2-2]
	\arrow[from=1-4, to=2-4]
	\arrow[from=1-1, to=1-2]
	\arrow[from=1-1, to=2-1]
\end{tikzcd}\]
is homotopic to the zero morphism.
Deduce the dimension of hom space between the two complexes in $\cK_{A_2}$.
\end{exercise}

The following operation on $\cK_{\Gamma}$ will be \emph{very important} for computation.
\begin{lemma}[Gaussian elimination \protect{\cite[Lemma 3.2]{bar-natan_burgos-soto_2014}}]\label{lem:gaussianelimination}
Consider a complex in $\cK_{\Gamma}$ of the form
\[
\cdots 
	\ra 
C 
	\xra{\begin{bmatrix}
		\alpha \\
		\beta
		\end{bmatrix}
		} 
b_1 \oplus D 
	\xra{\begin{bmatrix}
		\phi & \delta \\
		\gamma & \epsilon
		\end{bmatrix}
		}
b_2 \oplus E 
	\xra{\begin{bmatrix}
		\mu & \nu
		\end{bmatrix}
		}
F 
	\ra 
\cdots
\]
with $\phi\colon b_1 \ra b_2$ an isomorphism in $\cA_{\Gamma}$.
Then the complex is homotopy equivalent to
\[
\cdots 
	\ra 
C 
	\xra{\beta
		} 
D 
	\xra{\epsilon - \gamma  \phi^{-1} \delta
		}
E 
	\xra{\nu
		}
F 
	\ra 
\cdots.
\]
\end{lemma}
\begin{exercise}
\begin{enumerate}[(i)]
\item Show that the complex $\cdots \ra 0 \ra A \xra{\id_A} A \ra 0 \ra \cdots$ is homotopy equivalent to the zero complex.
Such a complex is said to be \emph{contractible}.
\item Note that one has to be extra careful in applying multiple Gaussian elimination simultaneously.
Consider the complex
\[
\cdots \ra 0 \ra X\oplus X \xra{\begin{bsmallmatrix} \id_X & \id_X \\ \id_X & \id_X \end{bsmallmatrix}} X \oplus X \ra 0 \ra \cdots.
\]
Is the complex above homotopy equivalent to the zero complex? (Hint: Is the matrix $\begin{bsmallmatrix} \id_X & \id_X \\ \id_X & \id_X \end{bsmallmatrix}$ invertible?) 
Apply Gaussian elimination \emph{carefully} to obtain the correct result.
\end{enumerate}
\end{exercise}
\begin{definition}
A complex $(C^j , d_j) \in \cK_{\Gamma}$ is called \emph{minimal} if there are no summands $X \overset{\oplus}{\subseteq} C^i$ and $Y \overset{\oplus}{\subseteq} C^{i+1}$ such that $d_j|_X\colon X \ra Y$ is an isomorphism in $\cA_{\Gamma}$.
\end{definition}
In other words, we have applied all possible Gaussian eliminations.

Note that the only isomorphisms between indecomposable objects in $\cA_{\Gamma}$ (i.e.\ objects of the form $P_s\<m\>$) are non-zero scalar multiples of the identity maps.
Since every complex in $\cK_{\Gamma}$ contains only a finite number of summands $P_s\<m\>$, there are only finitely many possible Gaussian eliminations that one can perform -- each Gaussian elimination process eliminates a finite (even) number of $P_s\<m\>$'s from the complex.
In other words, every complex in $\cK_{\Gamma}$ is homotopy equivalent to some complex which is minimal.
One might worry that different sequences of Gaussian eliminations will produce different minimal complexes, but the following result ensures that they will all be isomorphic \emph{as complexes} (no homotopy required).
\begin{proposition}[\protect{\cite[Section 6.1]{EW_hodge}}]\label{prop:minimaliso}
Let $M^\bullet, M'^\bullet$ be two minimal complexes of $C^\bullet$ obtained through Gaussian eliminations.
Then $M^\bullet$ and $M'^\bullet$ are isomorphic as complexes (no homotopy required).
\end{proposition}
\begin{proof}
Let $\cA_\Gamma^{ss}$ denote the (quotient) category with the same objects in $\cA_\Gamma$, but all morphisms of the form $(i|j)$ and $X_i$ are further identified with zero (so the only morphisms that survive are scalar multiplies of the identity maps $e_i$'s).
One can check that the canonical functor $\cF\colon \cA_\Gamma \ra \cA_\Gamma^{ss}$ that is the identity on objects and sends morphisms to their equivalence classes is well-defined\footnote{In other words, the equivalence relation generated by $(i|j) \sim 0$ and $X_i \sim 0$ forms an ideal in $\cA_\Gamma$.}.

With $\cK_\Gamma^{ss}$ denoting the homotopy category of complexes in $\cA_\Gamma^{ss}$, the functor above extends immediately to $\wt{\cF}\colon \cK_\Gamma \ra \cK_\Gamma^{ss}$ by applying $\cF$ to the complexes.
Note that any minimal complex $M^\bullet$ is mapped to a complex with zero differential under $\wt{\cF}$.
So if $f_\bullet\colon M^\bullet \ra M'^\bullet$ is a homotopy equivalence between two minimal complexes, then $\wt{\cF}(f_\bullet)\colon \wt{\cF}(M^\bullet) \ra \wt{\cF}(M'^\bullet)$ is a homotopy equivalence between two complexes with zero differential.
In particular, $\wt{\cF}(f_\bullet)$ is actually an isomorphism of complexes (no homotopy required), and so is $f_\bullet$.
\end{proof}
In particular, any two minimal complexes of the same complex in $\cK_\Gamma$ can only differ in their differentials up to invertible matrices.

\section{The Artin group action on \texorpdfstring{$\cK_{\Gamma}$}{K}} \label{sec:artinact}
For each $s \in \Gamma_0$ and each complex $C^\bullet\coloneqq (C^j,d_j) \in \cK_\Gamma$, we define a complex $\sigma_s(C^\bullet) \in \cK_{\Gamma}$ as follows.
Let $a(m,j) \coloneqq \dim \Hom_{\cA_{\Gamma}}(P_s\<m\>, C^j)$.
This is the same as counting the total number of length $(m-n)$ admissible walks from $s$ to $t$ for each $P_t\<n\>$ appearing as a summand of $C^j$.
Above each cohomological degree $j$ component of $C^\bullet$, we attach $\oplus_{m \in \Z} P_s\<m\>^{\oplus a(m,j)}$ as follows
\[
\begin{tikzcd}
\cdots
	& \oplus_{m \in \Z} P_s\<m\>^{\oplus a(m,-1)} \ar[d]
	& \oplus_{m \in \Z} P_s\<m\>^{\oplus a(m, 0)} \ar[d]
	& \oplus_{m \in \Z} P_s\<m\>^{\oplus a(m, 1)} \ar[d]
	& \cdots \\
\cdots \ar[r, "d_{-2}"] & C^{-1} \ar[r, "d_{-1}"] & C^0 \ar[r, "d_0"] & C^1 \ar[r, "d_1"]  & \cdots 
\end{tikzcd}
\]
with vertical maps from each summand $P_s\<m\>$ given by the standard basis elements of $\Hom_{\cA_\Gamma}(P_s\<m\>, C^j)$ (i.e.\ admissible walks) into the corresponding $C^j$.
The non-zero morphisms of the composition
\[
\begin{tikzcd}
\oplus_{m \in \Z} P_s\<m\>^{\oplus a(m,j)} \ar[d] \\
C^j \ar[r, "d_j"] & C^{j+1}
\end{tikzcd}
\]
are again the admissible walks into summands of $C^{j+1}$ (possibly up to rescaling).
As such, there is a unique way of choosing (appropriate scalar multiples of) identity maps (i.e.\ $e_s$) between the $P_s\<m\>$'s as the horizontal maps making the following diagram commutative:
\begin{equation} \label{eq:addingPabove}
\begin{tikzcd}
\cdots \ar[r]
	& \oplus_{m \in \Z} P_s\<m\>^{\oplus a(m,-1)} \ar[r] \ar[d]
	& \oplus_{m \in \Z} P_s\<m\>^{\oplus a(m, 0)} \ar[r] \ar[d]
	& \oplus_{m \in \Z} P_s\<m\>^{\oplus a(m, 1)} \ar[r] \ar[d]
	& \cdots \\
\cdots \ar[r, "d_{-2}"] & C^{-1} \ar[r, "d_{-1}"] & C^0 \ar[r, "d_0"] & C^1 \ar[r, "d_1"]  & \cdots 
\end{tikzcd}
\end{equation}
One can check that the first row is indeed a complex (i.e.\ horizontal maps square to zero) and it follows from the commutativity of the squares that the vertical morphisms define a morphism of complexes.
The cone of \eqref{eq:addingPabove} is the complex $\sigma_s(C^\bullet)$.
This construction is a mouthful, but the following example should be enlightening.
\begin{example}
Let $\Gamma = A_2 = \begin{tikzcd}[column sep=small] 1 \ar[r,no head] & 2 \end{tikzcd}$.
Consider the complex 
\[
C^\bullet\coloneqq 
\begin{tikzcd}
	0 & {P_1\<1\>} & {P_2} & 0 \\
	& {P_1\<-1\>}
	\arrow[from=1-2, to=1-3, "-(1|2)"]
	\arrow["\oplus"{description}, draw=none, from=1-2, to=2-2]
	\arrow[from=1-1, to=1-2]
	\arrow[from=1-3, to=1-4]
	\arrow[from=1-1, to=2-2]
\end{tikzcd}
\]
(The map from $P_1\<-1\>$ to $P_2$ is zero and the arrow is omitted; see \cref{rem:complexnotation}.)
Then  $\sigma_1(C^\bullet)$ is the cone of the following diagram
\[\begin{tikzcd}
	0 & {P_1\<1\>\oplus P_1\<3\>} & {P_1\<1\>} & 0 \\
	& {P_1\<-1\> \oplus P_1\<1\>} \\
	\\
	0 & {P_1\<1\>} & {P_2} & 0 \\
	& {P_1\<-1\>}
	\arrow["{{{{{-(1|2)}}}}}", from=4-2, to=4-3]
	\arrow["\oplus"{description}, draw=none, from=4-2, to=5-2]
	\arrow[from=4-1, to=4-2]
	\arrow[from=4-3, to=4-4]
	\arrow[from=4-1, to=5-2]
	\arrow[from=1-1, to=1-2]
	\arrow[from=1-1, to=4-1]
	\arrow[from=1-1, to=2-2]
	\arrow["\oplus"{description}, draw=none, from=1-2, to=2-2]
	\arrow["{{{[e_1 \ X_1]}}}"'{pos=0.4}, shift right=3, curve={height=24pt}, from=2-2, to=5-2]
	\arrow["{{{[e_1 \ X_1]}}}"{pos=0.5}, shift left=4, curve={height=-35pt}, from=1-2, to=4-2]
	\arrow["{{{{[-e_1 \ 0]}}}}", from=1-2, to=1-3]
	\arrow[from=1-3, to=1-4]
	\arrow["{{{(1|2)}}}", from=1-3, to=4-3]
	\arrow[from=1-4, to=4-4]
\end{tikzcd}.\]
\end{example}
\begin{exercise}
Compute the minimal complex of $\sigma_1(C^\bullet)$ above.
\end{exercise}

Given a morphism of complexes $f_{\bullet}\colon C^\bullet \ra D^\bullet$, we also have an induced morphism $\sigma_s(f_\bullet)\colon \sigma_s(C^\bullet) \ra \sigma_s(D^\bullet)$ defined as follows.
Before taking cones (i.e.\ \eqref{eq:addingPabove}), focusing on each column $i$, each $f_i\colon C^i \ra D^i$ induces a commutative square
\[\begin{tikzcd}[column sep=small]
	{\oplus_{m \in \bbZ} P_s\<m\>^{\oplus a_C(m,i)}} \\
	{C^i} & {\oplus_{m \in \bbZ} P_s\<m\>^{\oplus a_D(m,i)}} \\
	& {D^i}
	\arrow[from=1-1, to=2-1]
	\arrow["{f_i}", from=2-1, to=3-2]
	\arrow[from=1-1, to=2-2]
	\arrow[from=2-2, to=3-2]
\end{tikzcd}
,
\]
where the morphisms between the summands $P_s\<m\>$'s are again (scalar multiples of) identity maps from each summand that make the diagram commutative.
One can check that after taking cones to obtain $\sigma_s(C^\bullet)$ and $\sigma_s(D^\bullet)$, this defines a morphism of complexes from $\sigma_s(C^\bullet)$ to $\sigma_s(D^\bullet)$.
In fact, the operation on objects and morphisms makes $\sigma_s$ a \emph{functor} from $\cK_{\Gamma}$ to itself (an \emph{endofunctor}).

There is also an inverse operation: for each $s \in \Gamma_0$ and each complex $C^\bullet\coloneqq (C^j,d_j) \in \cK_\Gamma$, we define a complex $\sigma^{-1}_s(C^\bullet) \in \cK_{\Gamma}$ as follows.
Let $a'(m,j)\coloneqq \dim \Hom_{\cA_{\Gamma}}(C^j,P_s\<m\>)$; this is instead counting the total number of length $(n-m)$ admissible walks from $t$ to $s$ for each $P_t\<n\>$ appearing as a summand of $C^j$.
The complex $\sigma^{-1}_s(C^\bullet)$ is then the cone of the following:
\begin{equation} \label{eq:addingPbelow}
\begin{tikzcd}
\cdots \ar[r, "d_{-2}"] 
	& C^{-1} \ar[r, "d_{-1}"] \ar[d] 
	& C^0    \ar[r, "d_0"] \ar[d]
	& C^1    \ar[r, "d_1"] \ar[d]
	& \cdots \\
\cdots \ar[r]
	& \oplus_{m \in \Z} P_s\<m\>^{\oplus a'(m,-1)} \ar[r] 
	& \oplus_{m \in \Z} P_s\<m\>^{\oplus a'(m, 0)} \ar[r] 
	& \oplus_{m \in \Z} P_s\<m\>^{\oplus a'(m, 1)} \ar[r] 
	& \cdots
\end{tikzcd}
\end{equation}
where the vertical maps are defined from each summand $P_t\<n\>$ of $C^j$ by the admissible walk into the corresponding summand $P_s\<m\>$, and the horizontal maps in the second row are again (scalar multiples of) identity maps between the $P_s\<m\>$'s making the diagram commutative.
The operation of $\sigma_s^{-1}$ on morphisms can be defined similarly to $\sigma_s$, which we shall leave as an exercise to the reader; so $\sigma_s^{-1}$ is also a functor from $\cK_{\Gamma}$ to itself.

The notations used were meant to be suggestive:
\begin{proposition}
For each $s \in \Gamma_0$ and each complex $C^\bullet\coloneqq (C^j,d_j) \in \cK_\Gamma$, the following complexes are isomorphic in $\cK_{\Gamma}$ (i.e.\ they are homotopy equivalent)
\[
\sigma_s \sigma^{-1}_s(C^\bullet) \cong C^\bullet \cong \sigma_s^{-1} \sigma_s(C^\bullet).
\]
In fact, the equation above holds as isomorphism of endofunctors of $\cK_{\Gamma}$.
\end{proposition}

Moreover, these endofunctors also satisfy the braiding relations
\begin{proposition}
Let $C^\bullet\coloneqq (C^j,d_j) \in \cK_\Gamma$ and let $s, t \in \Gamma_0$.
\begin{itemize}
\item If $(s,t) \in \Gamma_1$, we have that
\[
\sigma_s \sigma_t \sigma_s(C^\bullet) \cong \sigma_t \sigma_s \sigma_t(C^\bullet) \in \cK_{\Gamma};
\]
\item otherwise $s$ and $t$ are not connected by an edge and we have that
\[
\sigma_s \sigma_t(C^\bullet) \cong \sigma_t \sigma_s(C^\bullet) \in \cK_{\Gamma}.
\]
\end{itemize}
In fact, the equations above hold as isomorphisms of endofunctors of $\cK_{\Gamma}$.
\end{proposition}
Proofs for the two results above for $\Gamma = A_n$ can be found in \cite{khovanov_seidel_2001}\footnote{The grading in \cite{khovanov_seidel_2001} is slightly different, but this does not affect the proof.}, which generalises to arbitrary $\Gamma$ with no difficulty (see also \cite{AK_link_hom, HueKho, thomas_seidel_2001}).

Combining the two proposition above, we have the following theorem.
\begin{theorem}\label{thm:braidaction}
We have a (weak) action of $\bbB(\Gamma)$ on $\cK_{\Gamma}$ via sending each standard generator $\sigma_s \in \bbB(\Gamma)$ to the invertible endofunctor $\sigma_s$.
\end{theorem}
The action of $\bbB(\Gamma)$ on $\cK_{\Gamma}$ is known to be faithful for $\Gamma$ of type A \cite{khovanov_seidel_2001, rouquier_zimmermann_2003}, type ADE \cite{brav_thomas_2010, LQ_ADE} and affine type A \cite{IUU_affineA}.
For $\bbB(\Gamma)$ a RAAG (where the construction of $\cA_\Gamma$ will require some simple modification to allow for edges labelled by $\infty$), see \cite{licata_2017_free}.

\begin{exercise}
Let $\Gamma = A_2 = \begin{tikzcd}[column sep=small] 1 \ar[r,no head] & 2 \end{tikzcd}$.
\begin{enumerate}[(i)] \label{ex:computeA2braid}
\item Compute $\sigma_i(P_j)$ for $i,j \in \{1,2\}$.
\item Show that $\sigma_1\sigma_2\sigma_1(P_1) \cong \sigma_2\sigma_1\sigma_2(P_1)$.
\end{enumerate}
\end{exercise}

\section{Triangulated structure of \texorpdfstring{$\cK_{\Gamma}$}{K}} \label{sec:triangulated}
Before we dive into stability conditions, we will need a categorical structure attached to $\cK_{\Gamma}$, which is the structure of a \emph{triangulated category}.
We won't flesh out all of the details; the keen reader is referred to \cite{neeman01}.
We do need, however, the notion of an \emph{exact triangle} (also known as a distinguished triangle).

Firstly, recall the construction of a cone of a morphism $f_\bullet \colon C^\bullet \ra D^\bullet$ of complexes.
Note that by construction, each cohomological degree component $j$ of the complex $\cone(f_\bullet)$ is a direct sum $C^{j+1} \oplus D^j$.

Let us define $C^\bullet[1]$ to be the complex $C^\bullet$ ``shifted to the left'' with \emph{negative differentials}; more precisely, the $j$th cohomological degree component of the complex $C^\bullet[1]$ is given by $C^{j+1}$ and the differentials maps $d_i$'s are all replaced with $-d_i$'s.
We have a chain of morphisms:
\begin{equation}\label{eq:conetriangle}
C^\bullet \xra{f_\bullet} D^\bullet \xra{\iota} \cone(f_\bullet) \xra{p} C^\bullet[1]
\end{equation}
with $\iota$ and $p$ denoting the canonical inclusion and projection respectively (one should check that these are indeed morphisms of complexes).
We call this a \emph{standard triangle} in $\cK_{\Gamma}$.
\begin{remark}
Note that there is a different \emph{internal grading shift} operation (functor) $\<\pm 1\>$ that one can define on $\cK_{\Gamma}$.
Given $P_j\<m\> \in \cA_{\Gamma}$, we define 
\[
(P_j\<m\>)\<\pm 1\> \coloneqq P_j\<m \pm 1\>.
\]
This carries forward to complexes in a similar fashion.
\end{remark}
\begin{definition}
An \emph{exact triangle} in $\cK_\Gamma$ is a chain of morphisms of complexes 
\[
X^\bullet \xra{f_\bullet} Y^\bullet \xra{g_\bullet} Z^\bullet \xra{h_\bullet} X[1]^\bullet
\]
that is isomorphic to some standard triangle $C^\bullet \xra{\theta_\bullet} D^\bullet \xra{\iota} \cone(f_\bullet) \xra{p} C^\bullet[1]$; this means that we have a commutative diagram in $\cK_\Gamma$:
\[
\begin{tikzcd}
X^\bullet \ar[r, "f_\bullet"] \ar[d,"\varphi_1"] &
	Y^\bullet \ar[r, "g_\bullet"] \ar[d,"\varphi_2"] &
	Z^\bullet \ar[r, "h_\bullet"] \ar[d,"\varphi_3"] &
	X[1] \ar[d,"\varphi_1{[1]}"] \\
C^\bullet \ar[r, "\theta_{\bullet}"] &
	D^\bullet \ar[r, "\iota"] &
	\cone(\varphi_{\bullet}) \ar[r, "p"] & 
	C^\bullet[1]
\end{tikzcd}
\]
with each $\varphi_i$ an isomorphism in $\cK_{\Gamma}$.
\end{definition}
\begin{remark}
The following notations are also commonly used to represent an exact triangle:
\[
    X \ra Y \ra Z \ra, \quad
        X \ra Y \ra Z \xra{+1}, \quad
            \begin{tikzcd}[column sep=small]
            X \ar[rr] & & Y \ar[ld] \\
            & Z \ar[ul, dashed]
            \end{tikzcd},
\]
where it is implied that the target of the last morphism must be $X[1]$ so it is omitted.
\end{remark}
\begin{exercise} \label{ex:exacttriangle}
\begin{enumerate}[(i)]
\item Show that $X^\bullet \xra{\id} X^\bullet \ra 0 \ra X[1]$ is an exact triangle.
\item Show that 
\[
D^\bullet \xra{\iota} \cone(f_\bullet) \xra{p} C^\bullet[1] \xra{-f_\bullet[1]} D^\bullet[1]
\]
is an exact triangle. (Hint: compute the cone of $\iota$.)
\item Deduce that $X^\bullet \xra{f} Y^\bullet \ra 0 \ra X^\bullet[1]$ is exact if and only if $X^\bullet \xra{f} Y^\bullet$ is an isomorphism in $\cK_\Gamma$.
\end{enumerate}
\end{exercise}
\begin{definition}
A functor $\cF\colon \cK \ra \cK'$ between triangulated categories is \emph{exact} if it sends exact triangles to exact triangles.
Note that this implies $\cF$ and $[1]$ commute (up to isomorphism of functors).
\end{definition}
\begin{exercise} \label{ex:actionexact}
Show that $\sigma_s\colon \cK_{\Gamma} \ra \cK_{\Gamma}$ is an exact (endo)functor.
\end{exercise}
\begin{definition}
The \emph{Grothendieck group} $K_0(\cK_{\Gamma})$ is the free abelian group generated by isomorphism classes of objects $[X^\bullet]$ in $\cK_{\Gamma}$, modulo the relations 
\[
[Y^\bullet] = [X^\bullet] + [Z^\bullet] \iff \exists \text{ an exact triangle of the form } X^\bullet \ra Y^\bullet \ra Z^\bullet \ra X^\bullet[1].
\]
\end{definition}
\begin{exercise}
Show that $[X^\bullet[1]] = -[X^\bullet]$. (Use \cref{ex:exacttriangle}. \emph{This is not to be confused with $[X^\bullet\<1\>]$.})
\end{exercise}
Let us now study the Grothendieck group $K_0(\cK_{\Gamma})$.
Firstly, note that any complex in $\cK_{\Gamma}$ is an iterated cone of $P_i\<m\>$'s.
Indeed, given a complex $C^\bullet \in \cK_{\Gamma}$ whose right-most non-zero cohomological degree component is $k$, we have the following exact (standard) triangle:
\[\begin{tikzcd}
	{C^k[-k]} & \cdots & 0 & 0 & {C^k} & 0 \\
	{C^\bullet} & \cdots & {C^{k-2}} & {C^{k-1}} & {C^k} & 0 \\
	{\cone(\iota)} & \cdots & {C^{k-2}} & {C^{k-1}} & 0 & 0 \\
	{C^k[-k+1]} & \cdots & 0 & {C^k} & 0 & 0
	\arrow[from=2-2, to=2-3]
	\arrow["{{d_{k-2}}}", from=2-3, to=2-4]
	\arrow["{{d_k}}", from=2-4, to=2-5]
	\arrow["\id"', from=2-3, to=3-3]
	\arrow["\id"', from=2-4, to=3-4]
	\arrow["\id"', from=1-5, to=2-5]
	\arrow[from=3-2, to=3-3]
	\arrow["{{d_{k-2}}}", from=3-3, to=3-4]
	\arrow[from=3-3, to=4-3]
	\arrow[from=3-4, to=3-5]
	\arrow[from=1-4, to=1-5]
	\arrow[from=1-3, to=1-4]
	\arrow[from=1-2, to=1-3]
	\arrow[from=1-3, to=2-3]
	\arrow[from=1-4, to=2-4]
	\arrow[from=2-5, to=3-5]
	\arrow[from=1-5, to=1-6]
	\arrow[from=2-5, to=2-6]
	\arrow[from=3-5, to=3-6]
	\arrow["{{=}}"{marking, allow upside down}, draw=none, from=1-1, to=1-2]
	\arrow["{{=}}"{marking, allow upside down}, draw=none, from=2-1, to=2-2]
	\arrow["{{\cong}}"{marking, allow upside down}, draw=none, from=3-1, to=3-2]
	\arrow["\iota"', from=1-1, to=2-1]
	\arrow[from=2-1, to=3-1]
	\arrow[from=3-1, to=4-1]
	\arrow["{{=}}"{marking, allow upside down}, draw=none, from=4-1, to=4-2]
	\arrow[from=4-2, to=4-3]
	\arrow[from=4-3, to=4-4]
	\arrow["{d_k}", from=3-4, to=4-4]
	\arrow[from=4-4, to=4-5]
	\arrow[from=4-5, to=4-6]
	\arrow[from=3-5, to=4-5]
	\arrow[from=1-6, to=2-6]
	\arrow[from=2-6, to=3-6]
	\arrow[from=3-6, to=4-6]
\end{tikzcd}\]
Note that for $A \in \cA_{\Gamma}$, we have used (and will continue using) $A[k]$ to denote the complex in $\cK_{\Gamma}$ concentrated in degree $-k$ (warning: not $k$).
When $k=0$, we shall abuse notation and simply write $A$ instead of $A[0]$; it will be clear from the context whether we are referring to $A \in \cA_{\Gamma}$ or $A \in \cK_{\Gamma}$.
\begin{exercise}
Show that for any $C^\bullet \in \cK_{\Gamma}$, the following holds in $K_0(\cK_{\Gamma})$:
\[
[C^\bullet] = \sum_{i \in \bbZ} (-1)^i[C^i].
\]
\end{exercise}
In fact, one can show that $K_0(\cK_{\Gamma})$ is a free abelian group with basis 
\[
\left\{ 
	[P_i\<m\>] \mid i \in \Gamma_0, \ m \in \bbZ
\right\}
\]
By defining $q^{\pm 1} \cdot [X^\bullet] = [X^\bullet \<\pm 1\>]$, we have that $K_0(\cK_{\Gamma})$ is a free $\bbZ[q^{\pm 1}]$-module with basis $\{[P_i] \mid i\in \Gamma_0\}$:
\begin{equation} \label{eq:grothendieckgrpq}
K_0(\cK_{\Gamma}) \cong \bigoplus_{i \in \Gamma_0} \bbZ[q^{\pm 1}]\{[P_i]\}
\end{equation}
Let $\Lambda_{\Gamma}$ denote the further quotient of $K_0(\cK_{\Gamma})$ that evaluates $q$ at $-1$, so that
\[
\Lambda_{\Gamma}= \bigoplus_{i \in \Gamma_0} \bbZ\{[P_i]\}.
\]
One should think of $\Lambda_\Gamma$ as the root lattice associated to the Coxeter (Weyl) group $\bbW(\Gamma)$.
Namely, by equipping $\Lambda_\Gamma$ with the symmetric bilinear product 
\[
(-,-): \Lambda_\Gamma \times \Lambda_\Gamma \ra \bbZ
\]
defined on the basis elements as follows (note that this bilinear form is twice the ``usual'' one: $(-,-) = 2\cB(-,-)$)
\[
([P_i], [P_j]) \coloneqq
    \begin{cases}
        2,  &\text{ if } i=j; \\
        -1, &\text{ if } (i,j) \in \Gamma_1; \\
        0,  &\text{ otherwise},
    \end{cases}
\]
we have an action of $\bbW(\Gamma)$ on $\Lambda_\Gamma$ given by
\[
s_i(v) = v - ([P_i], v)[P_i],
\]
where $s_i \in \bbW(\Gamma)$ denote the standard generators in $\bbW(\Gamma)$.
This action is the faithful \emph{geometric representation} of $\bbW(\Gamma)$ (or the Tits' representation).

As the action of $\bbB(\Gamma)$ on $\cK_{\Gamma}$ is exact, it descends onto an action on the Grothendieck group $K_0(\cK_{\Gamma})$, where
\[
\sigma_i([X^\bullet]) \coloneqq [\sigma_i(X^\bullet)].
\]
Furthermore, the action descends onto the quotient $\Lambda_{\Gamma}$, since the functors $\sigma_i$ and $\<\pm 1\>$ commute (check this!).
\begin{exercise}
Use the computations from \cref{ex:computeA2braid} to show that the action of $\bbB(A_2)$ on $K_0(\cK_{A_2})$ is isomorphic to the reduced Burau representation of the braid group (on 3 strands).
Moreover, after evaluating $q=-1$, show that the action on $\Lambda_\Gamma$ factors through the standard geometric representation of $\bbW(A_2) \cong S_3$.
\end{exercise}
For general Coxeter graphs $\Gamma$, the result above also follows from a direct computation on the basis elements $[P_i]$ of $K_0(\cK_{\Gamma})$, which we record here.
\begin{proposition}\label{prop:coxeterongrothendieck}
The induced action of $\bbB(\Gamma)$ on the quotient $\Lambda_{\Gamma} \cong K_0(\cK_{\Gamma})|_{q=-1}$ factors through the geometric representation of $\bbW(\Gamma)$ on $\Lambda_\Gamma$, i.e.\ the following diagram commutes:
\[
\begin{tikzcd}
K_0(\cK) \ar[r, "q=-1"] \ar[d,"\sigma_i"] & \Lambda_\Gamma \ar[d,"s_i"]\\
K_0(\cK_{\Gamma}) \ar[r, "q=-1"] & \Lambda_\Gamma
\end{tikzcd}.
\]
In particular, $\sigma_i^2 \in \bbB(\Gamma)$ acts as the identity on $\Lambda_{\Gamma}$.
\end{proposition}

\section{Hearts and linear complexes} \label{sec:heartlinear}
In this section, we introduce the notion of hearts and linear complexes, which will be used in the study of stability conditions in later sections.

\begin{definition}[\protect{\cite{BBD_faisceaux}}]
A \textit{heart} (of a bounded $t$-structure) in a triangulated category $\cD$ is a full additive subcategory $\cH$ such that:
	\begin{enumerate}[(1)]
		\item If $k_1>k_2$ then $\Hom_{\cD}(H_1[k_1],H_2[k_2])=0$ for all objects $H_1, H_2$ in $\cH$.
		\item For any object $X$ in $\cD$ there is a sequence of exact triangles called the \emph{heart filtration} with respect to $\cH$:
		\begin{equation} \label{eq:heartfiltration}
\begin{tikzcd}[column sep = 3mm]
0
	\ar[rr] &
{}
	{} &
X_1
	\ar[rr] \ar[dl]&
{}
	{} &	
X_2
	\ar[rr] \ar[dl]&
{}
	{} &	
\cdots
	\ar[rr] &
{}
	{} &
X_{m-1}
	\ar[rr] &
{}
	{} &
X_m = X
	\ar[dl] \\
{}
	{} &
A_1
	\ar[lu, dashed] &
{}
	{} &
A_2
	\ar[lu, dashed] &
{}
	{} &
{}
	{} &
{}
	{} &
{}
	{} &
{}
	{} &
A_m
	\ar[lu, dashed] &
\end{tikzcd}
\end{equation}
		such that $A_i\in \cH[k_i]$ and $k_1>k_2>\cdots > k_m$.
	\end{enumerate}
Note that the hom vanishing condition in (1) implies that the heart filtration is unique up to isomorphism.
\end{definition}
One should think of a choice of heart as a way to ``cut'' objects in $\cD$ uniquely into a filtration of ``shifted heart pieces''.

In fact, $\cH$ is an \emph{abelian category}; once again, we omit the full definition of an abelian category, but we point out that this means in particular that $\cH$ has \emph{exact sequences}, which are given by
\begin{equation}\label{eq:exactseqtriangle}
0 \ra X \xra{f} Y \xra{g} Z \ra 0 \text{ is exact in } \cH \iff X \xra{f} Y \xra{g} Z \ra X[1] \text{ is an exact triangle in } \cD.
\end{equation}
\begin{exercise} \label{ex:autoheartisheart}
Suppose $\cH$ is a heart in $\cD$ and $\sigma\colon \cD \ra \cD$ is an exact endofunctor that is invertible (an autoequivalence).
Show that $\sigma(\cH) \subseteq \cD$ is again a heart in $\cD$.
\end{exercise}
Our category $\cK_\Gamma$ comes with a standard heart that can be defined as follows.
\begin{definition}
Let $C^\bullet$ be a minimal complex in $\cK_{\Gamma}$.
We say that $C^\bullet$ is \emph{linear} if for each $j \in \bbZ$, every indecomposable summand of $C^j$ is of the form $P_i\<-j\>$ for some $i \in \Gamma_0$.
Similarly, a complex in $\cK_{\Gamma}$ is linear if it has a minimal complex that is linear (this is independent of the choice of minimal complex by \cref{prop:minimaliso}).
The smallest full additive subcategory of $\cK_{\Gamma}$ containing all linear complexes will be denoted by $\cH_{lin}$.
The category $\cH_{lin}$ forms a heart in $\cK_{\Gamma}$, and shall be called the \emph{heart of linear complexes} (see \cref{ex:heartlinearcomplex}).
\end{definition}
Note that $\cH_{lin}$ is closed under \emph{linear shifts} $\<m\>[m]$ for all $m \in \bbZ$; however, it is not closed under $\<m\>$ nor $[m]$ separately (except for $m=0$, of course).
\begin{example}[and non-example]
For $\Gamma=A_2$, the complex 
\[
0 \ra P_1\<1\> \xra{(1|2)} P_2 \ra 0 \qquad \text{(linear)}
\]
with $P_2$ sitting in cohomological degree zero is linear, whereas 
\[
\sigma_1(0 \ra P_1\<1\> \xra{(1|2)} P_2 \ra 0) \cong 0 \ra P_1\<3\> \xra{X_1} P_1\<1\> \xra{(1|2)} P_2 \ra 0 \qquad \text{(not linear)}
\]
is not.
\end{example}
More generally, the morphisms appearing in the differentials of a minimal complex must be scalar multiples (including zero) of length one admissible walks $(i|j)$.
\begin{exercise}\label{ex:heartlinearcomplex}
\begin{enumerate}[(i)]
\item Show that up to linear shifts $\<m\>[m]$, the indecomposable, minimal linear complexes of $\cK_{A_2}$ are (up to homotopy equivalence):
\begin{equation}\label{eq:indecomplinearA2}
P_1, P_2, \sigma_1(P_2) = 0 \ra P_1\<1\> \xra{(1|2)} P_2 \ra 0, \sigma_2(P_1) = 0 \ra P_2\<1\> \xra{(2|1)} P_1 \ra 0.
\end{equation}
Using this, write out the heart filtration with respect to $\cH_{lin}$ for the (non-linear) complex
\[
\sigma_1(0 \ra P_1\<1\> \xra{(1|2)} P_2 \ra 0) \cong 0 \ra P_1\<3\> \xra{X_1} P_1\<1\> \xra{(1|2)} P_2 \ra 0.
\]
\item Check that $\cH_{lin}$ is indeed a heart in $\cK_{\Gamma}$ (you may choose to only do this for $\Gamma = A_2$).
\item Using \cref{ex:autoheartisheart}, we know that $\sigma_1(\cH_{lin})$ is a (different) heart in $\cK_{\Gamma}$.
For example, the complex $\sigma_1(0 \ra P_1\<1\> \xra{(1|2)} P_2 \ra 0)$ is, by definition, in the heart $\sigma_1(\cH_{lin})$ and so its heart filtration \emph{with respect to $\sigma_1(\cH_{lin})$} is trivial. 
For the four indecomposable objects in \eqref{eq:indecomplinearA2}, work out their heart filtrations with respect to this new heart $\sigma_1(\cH_{lin})$.
\end{enumerate}
\end{exercise}
\section{The moduli space of stability conditions \texorpdfstring{$\Stab(\cK_{\Gamma})$}{Stab(K)}} \label{sec:stab}
We briefly introduce the notion of Bridgeland's stability conditions \cite{bridgeland_2007} in this section, focusing mainly on our triangulated category of interest $\cK_\Gamma$. 
We refer the interested reader to \cite{bayer11} for a slightly more expanded introduction to stability conditions.

From here on, we will fix the following conventions.
$\cD$ will always denote a triangulated category (you may take $\cD = \cK_{\Gamma}$).
We assume throughout that we have fixed a quotient $v\colon K_0(\cD) \ra \Lambda$ onto some finite rank free abelian group $\Lambda$; in our case $\cD = \cK_{\Gamma}$, we have $\Lambda \coloneqq \Lambda_{\Gamma}$ with $v$ given by evaluating $q=-1$.
Given any morphism of abelian groups $Z\in \Hom_{\Z}(\Lambda,\C)$, for each object $X \in \cD$, we will also write
\[
Z(X) \coloneqq Z(v([X])).
\]
\begin{definition}[\protect{\cite{bridgeland_2007}}] \label{defn:heartstability}
Let $\cH$ be a heart in $\cD$.
A \emph{stability function} on $\cH$ is a group homomorphism $Z\in \Hom_{\Z}(\Lambda,\C)$ such that for all non-zero objects $E$ in $\cH \subseteq \cD$, we have $Z(E)$ lying in the strict upper-half plane union the strictly negative real line:
\[
\mathbb{H} \cup \R_{<0} = \{re^{i\pi\phi} : r \in \R_{> 0} \text{ and } 0 < \phi \leq 1\}.
\]
Expressing $Z(E) = m(E)e^{i\pi\phi}$ as above, we call $\phi$ the \emph{phase} of $E$, which we denote by $\phi(E)$.
We say that an object $E\in \cH$ is $Z$-\emph{semistable} (resp.\ $Z$-\emph{stable}) if every object $U \in \cH$ fitting into an exact sequence $0 \ra U \ra E \ra V \ra 0$ in $\cH$ satisfies $\phi(U) \leq$ (resp.\ $<$) $\phi(E)$. \\
A stability function is said to have the \emph{Harder-Narasimhan (HN) property} if every object $A\in \cH$ has a filtration by exact sequences (viewed as exact triangles in $\cD$; see \eqref{eq:exactseqtriangle})
\[
\begin{tikzcd}[column sep = 3mm]
0
	\ar[rr] &
{}
	{} &
A_1
	\ar[rr] \ar[dl]&
{}
	{} &	
A_2
	\ar[rr] \ar[dl]&
{}
	{} &	
\cdots
	\ar[rr] &
{}
	{} &
A_{m-1}
	\ar[rr] &
{}
	{} &
A_m = A
	\ar[dl] \\
{}
	{} &
S_1
	\ar[lu, dashed] &
{}
	{} &
S_2
	\ar[lu, dashed] &
{}
	{} &
{}
	{} &
{}
	{} &
{}
	{} &
{}
	{} &
S_m
	\ar[lu, dashed] &
\end{tikzcd}
\]
such that all $S_i \in \cH$ are semistable and $\phi(S_i) > \phi(S_{i+1})$.
Note that this filtration, if it exists, is unique up to isomorphism.
Hence we can define $\phi_+(A) \coloneqq \phi(S_1)$ and $\phi_-(A) \coloneqq \phi(S_m)$.
\end{definition}
\begin{remark}\label{rem:HNautomatic}
The HN property is automatically satisfied if the heart is a finite-length abelian category (i.e.\ every object in the heart has a Jordan--H\"older filtration of finite length).
The heart of linear complexes $\cH_{lin} \subseteq \cK_{\Gamma}$ is finite-length for all $\Gamma$ -- and so is $\sigma(\cH_{lin})$ for any invertible, exact endofunctor $\sigma$.
As such, any stability function on $\sigma(\cH_{lin})$ will satisfy the HN property.
\end{remark}

\begin{definition}[\protect{\cite{bridgeland_2007}}]\label{defn:stabcondusingheart}
A \emph{stability condition} of a triangulated category $\cD$ is a choice of a heart $\cH \subseteq \cD$ together with a stability function on $\cH$ satisfying the HN property.
\end{definition}

\begin{example} \label{eg:A2stab}
Consider the heart of linear complexes $\cH_{lin}$ in $\cK_{A_2}$.
Any group homomorphism $Z\in \Hom_{\bbZ}(\Lambda_{A_2}, \bbC)$ is uniquely determined by $Z(P_1)$ and $Z(P_2)$.
Moreover, $Z$ is a stability function on $\cH_{lin}$ if and only if
\[
Z(P_1), Z(P_2) \in \bbH \cup \bbR_{<0};
\]
the requirement that all non-zero objects in $\cH_{lin}$ are mapped to $\bbH \cup \bbR_{<0}$ follows immediately. Indeed, the $Z$-image of all objects in $\cH_{lin}$ live in the convex cone of $Z(P_i)$.
As discussed in \cref{rem:HNautomatic}, any stability function on $\cH_{lin}$ will satisfy the HN property, and thus defines a stability condition on $\cK_{A_2}$.
Consider the stability function $Z$ defined as follows:
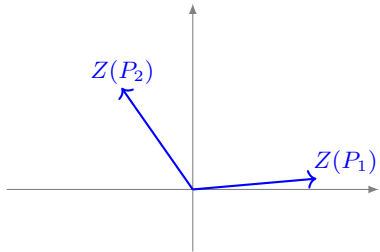
\[
\begin{tikzpicture}[scale = 1.65]
    \coordinate (Origin)   at (0,0);
    \coordinate (XAxisMin) at (-1.5,0);
    \coordinate (XAxisMax) at (1.5,0);
    \coordinate (YAxisMin) at (0,-.5);
    \coordinate (YAxisMax) at (0,1.5);
    \coordinate (na1)      at (-1,0);

    \draw [thin, gray,-latex] (XAxisMin) -- (XAxisMax);
    \draw [thin, gray,-latex] (YAxisMin) -- (YAxisMax);

	\draw[blue, thick, ->] (0,0) -- ++(5:1);
	\node[blue] (P1) at ($(Origin) + (10:1.25)$) 
		{\scriptsize $Z(P_1)$};
	
	\draw[blue, thick, ->] (0,0) -- ++(125:1);
	\node[blue] (P2) at ($(Origin) + (121:1.1)$) 
		{\scriptsize $Z(P_2)$};
  \end{tikzpicture}
\]
In this case, up to linear shifts $\<k\>[k]$, we have 3 $Z$-(semi)stable objects that are minimal: $P_1$, $P_2$ and $0 \ra P_2\<1\> \xra{(2|1)} P_1 \ra 0$ (where $P_1$ sits in cohomological degree zero).
\end{example}
\begin{exercise}
Now define the stability function $Z'$ with $Z'(P_1) = Z(P_2)$ and $Z'(P_2) = Z(P_1)$. What are the $Z'$-(semistable) objects?
What about $Z''$ with $Z''(P_1) = Z''(P_2) \in \bbH \cup \bbR_{<0}$?
\end{exercise}

\begin{definition}
Let $Z$ be a stability function on a heart $\cH$ of $\cD$ satisfying the HN property; i.e.\ we have a stability condition on $\cD$.
For each $0 < \phi \leq 1$, let $\cP(\phi)$ denote the full additive subcategory of $\cH \subseteq \cD$ containing all the $Z$-semistable objects in $\cH$ of phase $\phi$; if there are no $Z$-semistable objects of phase $\phi$, then $\cP(\phi)$ contains only the zero object.
For all other $\phi \in \bbR$, we define $\cP(\phi):= \cP(\phi-m)[m]$ with (the unique) $m \in \bbZ$ such that $0 < \phi-m \leq 1$.
The collection of full additive subcategories $\cP = \{\cP(\phi) \subseteq \cD \mid \phi \in \bbR\}$ of the stability condition is called a \emph{slicing} of the stability condition.
\end{definition}

One can check that a slicing satisfies the following:
\begin{enumerate}
\item $\cP(\phi + 1) = \cP(\phi)[1]$;
\item $\Hom( \cP(\phi_1), \cP(\phi_2)) = 0$ for all $\phi_1 > \phi_2$; and
\item for each object $X \in \cD$, there exists a filtration of $X$:
\[
\begin{tikzcd}[column sep = 3mm]
0
	\ar[rr] &
{}
	{} &
X_1
	\ar[rr] \ar[dl]&
{}
	{} &	
X_2
	\ar[rr] \ar[dl]&
{}
	{} &	
\cdots
	\ar[rr] &
{}
	{} &
X_{m-1}
	\ar[rr] &
{}
	{} &
X_m = X
	\ar[dl] \\
{}
	{} &
E_1
	\ar[lu, dashed] &
{}
	{} &
E_2
	\ar[lu, dashed] &
{}
	{} &
{}
	{} &
{}
	{} &
{}
	{} &
{}
	{} &
E_m
	\ar[lu, dashed] &
\end{tikzcd}
\]
such that $E_i \in \cP(\phi_i)$ and $\phi_1 > \phi_2 > \cdots > \phi_m$.
\end{enumerate}
The filtration in (iii) is obtained by refining the heart filtration of $X$ (from \eqref{eq:heartfiltration}) using the HN filtration of each $A_i$ (see \cref{app:refinement} for more details about refinement); this filtration is called the \emph{Harder-Narasimhan (HN) filtration} of $X \in \cD$ and it is also unique up to isomorphism.
As in the case for $\cH$, for objects $X \in \cD$, we define $\phi_+(X) \coloneqq \phi_1$ and $\phi_-(X) \coloneqq \phi_m$.
\begin{remark}
In fact, the definition of a slicing of $\cD$ is actually just a collection of full additive subcategory $\cP = \{\cP(\phi) \subseteq \cD \mid \phi \in \bbR\}$ of $\cD$ satisfying conditions (i)--(iii) above.
Bridgeland showed \cite[Proposition 5.3]{bridgeland_2007} that one can equivalently define a stability condition as a pair $(\cP, Z)$ of a slicing $\cP$ and a group homomorphism $Z \in \Hom_{\bbZ}(\Lambda, \bbC)$ called the central charge, which satisfies the condition that any non-zero $E \in \cP(\phi)$ has its central charge given by
\[
Z(E) = m(E)e^{i\pi\phi} \quad \text{ for some } m(E) \in \R_{>0}.
\]
(In fact, this was the original definition.)
In particular, one recovers the heart by taking the smallest additive subcategory (usually denoted by $\cP(0,1]$) containing all objects $A \in \cD$ such that the phases of the semistable pieces all live in $(0, 1]$.
The stability function is then given by the central charge $Z$.
\end{remark}
\begin{definition}
Let $\Stab(\cD)$ denote the set of stability conditions\footnote{Technically, for the results that follow, one requires a technical assumption on stability conditions known as \emph{locally-finiteness}, or the \emph{support property} (stronger). In these notes, we will only discuss stability conditions for $\Gamma$ being ADE, where these conditions will be superfluous.} on $\cD$.
Viewing elements in $\Stab(\cD)$ as a pair $(\cP, Z)$ consisting of a slicing $\cP$ and a stability function $Z \in \Hom_{\bbZ}(\Lambda, \bbC)$ on $\cH$, we shall define a topology on $\Stab(\cD)$ as follows.
Let $(\cP, Z)$ and $(\cQ, Z')$ be two stability conditions.
We define a generalised metric\footnote{A generalised metric $d(-,-)$ satisfies all properties of a metric except possibly taking infinite value. In particular, it defines a topology in the usual way, where $d(\xi, \varsigma) = \infty$ implies $\xi$ and $\varsigma$ live on different connected components.} on the set of slicings:
\[
d(\cP,\cQ):= \sup_{0\neq E \in \cD} \left\{
	|\phi_-^\cP(E) - \phi_-^{\cQ}(E)|, |\phi_+^\cP(E) - \phi_+^{\cQ}(E)|
	\right\}.
\]
This induces a topology on the set of slicings, and $\Hom_{\bbZ}(\Lambda, \bbC) \cong \bbC^{m}$ ($m$ is the rank of $\Lambda$) is equipped with the standard topology as $\bbC$-vector space.
The topology on $\Stab(\cD)$ is defined to be the coarsest topology such that both forgetful maps $(\cP, Z) \mapsto \cP$ and $(\cP, Z) \mapsto Z$ are continuous.
\end{definition}
The following lemma will be useful in patching together subsets of stability conditions:
\begin{lemma}\label{lem:uniqueslicing<1}
Let $(\cP, Z)$ and $(\cQ, Z)$ be two stability conditions with the same central charge $Z$. 
If $d(\cP,\cQ) < 1$, then $\cP = \cQ$.
\end{lemma}
Note that if $d(\cP, \cQ) \geq 1$, then the above may not hold:
\begin{exercise}
Let $Z$ be any stability function on $\cH_{lin} \subseteq \cK_{A_2}$.
\begin{enumerate}
\item Show that $Z$ defines a stability function on the heart $\sigma_i^2(\cH_{lin})$. (Hint: what is the action of $\sigma_i^2$ on $K_0(\cK_\Gamma)$?)
\item Let $\cP$ and $\cQ$ be the slicings defined by the (same) stability function $Z$ on $\cH_{lin}$ and $\sigma_i^2(\cH_{lin})$ respectively. Show that $d(\cP,\cQ) > 1$. (Hint: Compare the phase of $P_1$ in $\cP$ and $\cQ$.)
\end{enumerate}
\end{exercise}
The following deformation result allows one to (locally) deform $Z$ to obtain new stability conditions:
\begin{lemma}[\protect{\cite[Theorem 7.1]{bridgeland_2007}}]\label{lem:deformation}
Let $\xi = (\cP, Z)$ be a stability condition. 
For all $0<\varepsilon<1/8$, if $W \in \Hom_{\Z}(\Lambda,\C)$ satisfies
\begin{equation} \label{eq:deformcondition}
|W(E) - Z(E)|<\sin(\varepsilon \pi)|Z(E)|
\end{equation}
for all non-zero $E \in \cP(\phi)$ and all $\phi \in \bbR$, then there exists a slicing $\cQ$ so that $\varsigma = (\cQ, W)$ forms a stability condition on $\cD$ with $d(\cP,\cQ) < \varepsilon$.
\end{lemma}
Note that \eqref{eq:deformcondition} is an open condition on $\Hom_{\bbZ}(\Lambda, \bbC)$, and it implies, in particular, that for any $E \in \cP(\phi)$ and any $\phi \in \bbR$:
\begin{itemize}
\item the phase of $W(E)$ and $Z(E)$ can only differ by at most $\varepsilon$; and
\item $W(E) \neq 0$.
\end{itemize}

Combining \cref{lem:uniqueslicing<1} and \cref{lem:deformation}, we obtain the main theorem of Bridgeland showing that $\Stab(\cD)$ forms a complex manifold.
\begin{theorem}[\protect{\cite[Theorem 1.2]{bridgeland_2007}}] \label{thm:stabmfld}
The space of stability conditions on $\cD$, $\Stab(\cD)$, is a complex manifold (of dimension equal the rank of $\Lambda$), with the forgetful map $\cZ$ defining the local homeomorphism 
\begin{align*}
\cZ\colon \Stab(\cD) &\ra \Hom_{\Z}(\Lambda,\C) \\
(\cP, Z) &\mapsto Z.
\end{align*}
\end{theorem}

\begin{example} \label{eg:deformP1}
Let us continue with \cref{eg:A2stab}, where we have drawn all the central charges for the (semi)stable objects in $\cH_{lin} \subseteq \cK_{A_2}$:
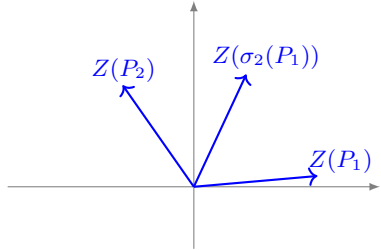
\[
\begin{tikzpicture}[scale = 1.65]
    \coordinate (Origin)   at (0,0);
    \coordinate (XAxisMin) at (-1.5,0);
    \coordinate (XAxisMax) at (1.5,0);
    \coordinate (YAxisMin) at (0,-.5);
    \coordinate (YAxisMax) at (0,1.5);
    \coordinate (na1)      at (-1,0);

    \draw [thin, gray,-latex] (XAxisMin) -- (XAxisMax);
    \draw [thin, gray,-latex] (YAxisMin) -- (YAxisMax);

	\draw[blue, thick, ->] (0,0) -- ++(5:1);
	\node[blue] (s2s1s2a1) at ($(Origin) + (10:1.2)$) 
		{\scriptsize $Z(P_1)$};
		
	\draw[blue, thick, ->] (0,0) -- ++(65:1);
	\node[blue] (a2) at ($(Origin) + (61:1.2)$) 
		{\scriptsize $Z(\sigma_2(P_1))$};
	
	\draw[blue, thick, ->] (0,0) -- ++(125:1);
	\node[blue] (a2) at ($(Origin) + (121:1.1)$) 
		{\scriptsize $Z(P_2)$};
  \end{tikzpicture}
\]
Suppose we now deform $Z$ slightly ($\varepsilon$ should be taken to be very small) to obtain a new central charge $Z'$ such that $Z'(P_1)$ now has non-positive imaginary part (and $Z'(P_2) = Z(P_2)$):
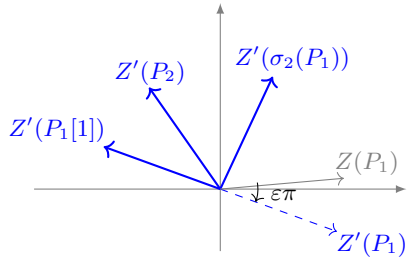
\[
\begin{tikzpicture}[scale = 1.65]
    \coordinate (Origin)   at (0,0);
    \coordinate (XAxisMin) at (-1.5,0);
    \coordinate (XAxisMax) at (1.5,0);
    \coordinate (YAxisMin) at (0,-.5);
    \coordinate (YAxisMax) at (0,1.5);

    \draw [thin, gray,-latex] (XAxisMin) -- (XAxisMax);
    \draw [thin, gray,-latex] (YAxisMin) -- (YAxisMax);

	\draw[blue, dashed, ->] (0,0) -- ++(-20:1);
	\node[blue] (P1) at ($(Origin) + (-20:1.3)$) 
		{\scriptsize $Z'(P_1)$};
		
	\draw[blue, thick, ->] (0,0) -- ++(65:1);
	\node[blue] (s2P1) at ($(Origin) + (61:1.2)$) 
		{\scriptsize $Z'(\sigma_2(P_1))$};
	
	\draw[blue, thick, ->] (0,0) -- ++(125:1);
	\node[blue] (P2) at ($(Origin) + (121:1.1)$) 
		{\scriptsize $Z'(P_2)$};
	
	\draw[blue, thick, ->] (0,0) -- ++(160:1);
	\node[blue] (P1shift) at ($(Origin) + (160:1.4)$) 
		{\scriptsize $Z'(P_1[1])$};
	
	\draw[thin, gray, ->] (0,0) -- ++(5:1);
	\node[gray] (oldP1) at ($(Origin) + (10:1.2)$) 
		{\scriptsize $Z(P_1)$};
	
	\pic [draw, <-, "\scriptsize $\varepsilon \pi$", angle eccentricity=1.7] {angle=P1--Origin--oldP1};
  \end{tikzpicture}
\]
Clearly, $Z'$ can no longer be a stability function on $\cH_{lin}$, since $Z'(P_2)$ no longer lives in $\bbH \cup \bbR_{<0}$; however, we see that $Z'(P_1[1])$ enters the picture.
More precisely, with $\cP$ denoting the slicing induced from the stability function $Z$ on $\cH_{lin}$, we claim that $Z'$ is a stability function on the heart $\sigma_1^{-1}(\cH_{lin})$, such that the induced slicing $\cQ$ satisfies $d(\cP, \cQ) < 1$. \cref{lem:uniqueslicing<1} then implies that this is the unique lift of $Z'$ from the deformation result \cref{lem:deformation}.

Let us first understand $\sigma_1^{-1}(\cH_{lin})$: its four indecomposable objects (again, up to linear shifts and homotopy equivalence) are:
\begin{align*}
\sigma_1^{-1}(P_1)\<-2\>[-2] &\cong P_1[1]; \\
\sigma_1^{-1}(P_2)\<1\>[1] &\cong 0 \ra P_2\<1\> \xra{(2|1)} P_1 \ra 0 \cong \sigma_2(P_1); \\
\sigma_1^{-1}(\sigma_1(P_2)) &\cong P_2 \\
\sigma_1^{-1}(\sigma_2(P_1)) &\cong 0 \ra P_2\<1\> \xra{(2|1)} P_1 \xra{X_1} P_1\<-2\> \ra 0.
\end{align*}
Only the first three are $Z'$-semistable; we have an exact sequence in $\sigma_1^{-1}(\cH_{lin})$:
\[
0\ra P_1[1]\<-2\>[-2] \ra \sigma_1^{-1}(\sigma_2(P_1)) \ra \sigma_2(P_1) \ra 0.
\] 
If we denote the phases $0< \phi_1, \phi_{21}, \phi_2 \leq 1$ such that:
\[
P_1 \in \cP(\phi_1), \quad \sigma_2(P_1) \in \cP(\phi_{21}), \quad P_2 \in \cP(\phi_2),
\]
we now get
\[
P_1[1] \in \cQ(\phi_1 + 1 - \varepsilon), \quad \sigma_2(P_1) \in \cP(\phi_{21} - \varepsilon'), \quad P_2 \in \cP(\phi_2)
\]
with $0 < \varepsilon' < \varepsilon$.
In particular, we see that $d(\cP, \cQ) = \epsilon < 1$.
\end{example}

\begin{example}[Falling into a hole]
Note that while deforming the central charge $Z$ of a stability condition $(\cP,Z)$, one has to avoid ``falling into holes'': we can not deform into a central charge $W$ with $W(E)=0$ for some $E \in \cP(\phi)$ -- this violates \eqref{eq:deformcondition} in \ref{lem:deformation}.
For example, if we started with $Z(P_2)$ very close to the negative real line, we may have deformed $Z$ into $Z'$ such that $Z'(P_1[1]) = Z'(P_2)$:
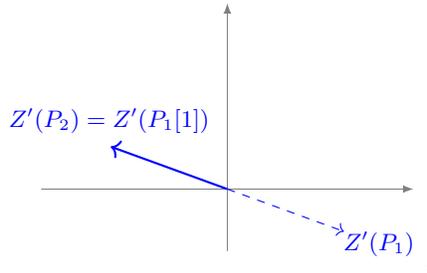
\[
\begin{tikzpicture}[scale = 1.65]
    \coordinate (Origin)   at (0,0);
    \coordinate (XAxisMin) at (-1.5,0);
    \coordinate (XAxisMax) at (1.5,0);
    \coordinate (YAxisMin) at (0,-.5);
    \coordinate (YAxisMax) at (0,1.5);

    \draw [thin, gray,-latex] (XAxisMin) -- (XAxisMax);
    \draw [thin, gray,-latex] (YAxisMin) -- (YAxisMax);

	\draw[blue, thick, ->] (0,0) -- ++(160:1);
	\node[blue] (P2) at ($(Origin) + (150:1.1)$) 
		{\scriptsize $Z'(P_2)=Z'(P_1[1])$};
	
	\draw[blue, dashed, ->] (0,0) -- ++(-20:1);
	\node[blue] (P1) at ($(Origin) + (-20:1.3)$) 
		{\scriptsize $Z'(P_1)$};
\end{tikzpicture},
\]
In this case, we would have $Z'(\sigma_2(P_1)) = Z'(P_2) + Z'(P_1) = 0$, which is not allowed since $\sigma_2(P_1)$ was in $\cP(\phi_{21})$.
In fact, we see that $Z'$ also violates the condition of being a stability function on $\sigma_1^{-1}(\cH_{lin})$!
\end{example}

\begin{exercise}\label{ex:crossnegative}
Start instead with the following stability function on $\cH_{lin}$:
\[
\begin{tikzpicture}[scale = 1.65]
    \coordinate (Origin)   at (0,0);
    \coordinate (XAxisMin) at (-1.5,0);
    \coordinate (XAxisMax) at (1.5,0);
    \coordinate (YAxisMin) at (0,-.5);
    \coordinate (YAxisMax) at (0,1.5);

    \draw [thin, gray,-latex] (XAxisMin) -- (XAxisMax);
    \draw [thin, gray,-latex] (YAxisMin) -- (YAxisMax);

		
	\draw[blue, thick, ->] (0,0) -- ++(60:1);
	\node[blue] (P1) at ($(Origin) + (60:1.1)$) 
		{\scriptsize $Z(P_1)$};
	
	\draw[blue, thick, ->] (0,0) -- ++(125:1);
	\node[blue] (s2P1) at ($(Origin) + (121:1.1)$) 
		{\scriptsize $Z(\sigma_2(P_1))$};
	
	\draw[blue, thick, ->] (0,0) -- ++(180:1);
	\node[blue] (P2) at ($(Origin) + (175:1.3)$) 
		{\scriptsize $Z(P_2)$};
\end{tikzpicture}
\]
Following \cref{eg:deformP1}, work out the the slight deformation where $Z(P_2)$ crosses the negative real line instead.
\end{exercise}

\section{Covering structure of  \texorpdfstring{$\Stab(\cK_{\Gamma})$}{Stab(K)}} \label{sec:coveringstab}
For this section, we will focus on the case where $\Gamma$ is an ADE Coxeter diagram.
The aim of this section is to show that $\cZ: \Stab(\cK_{\Gamma}) \ra \Hom_\bbZ(\Lambda_\Gamma, \bbC)$ is a covering map onto its image, where its image is exactly the hyperplane complement associated to the Coxeter group $\bbW(\Gamma)$ (through the contragradient representation).
The proofs presented here follow (an appropriate modification of) the proofs in \cite{bridgeland2009kleinian,ikeda2014stability}.
The upshot is that general results about the triangulated category $\cK_\Gamma$ will imply that $\Stab(\cK_\Gamma)$ is a contractible space, which proves that $\cZ$ is actually a universal cover and moreover (re-)proves the $K(\pi,1)$ conjecture for $\Gamma$ of ADE types.
We note here that the proof obtained will be independent of Deligne's original proof \cite{Deligne_72}.

Recall from \cref{thm:braidaction} and \cref{ex:actionexact} that we have an action of $\bbB(\Gamma)$ on $\cK_{\Gamma}$ by invertible, exact endofunctors.
Moreover, for each $\beta \in \bbB(\Gamma)$, $\beta(\cH_{lin})$ is again a heart in $\cK_\Gamma$, where any stability function on it would define a stability condition on $\cK_{\Gamma}$ (HN property comes for free in our setting).
More generally, one can define an action of $\beta \in \bbB(\Gamma)$ on $\Stab(\cK_\Gamma)$ as follows.
\begin{definition}\label{defn:actionstab}
Let $\beta \in \bbB(\Gamma)$.
For each stability condition $\varsigma = (\cP,Z)$ defined by a stability function $Z$ on some heart $\cH \subset \cK_{\Gamma}$, we define $\beta \cdot \varsigma$ as the stability condition defined by the stability function $Z \circ \beta^{-1}$ on the the heart $\beta(\cH)$. Equivalently, $\beta \cdot \varsigma \coloneqq (\beta\cdot \cP, Z \circ \beta^{-1})$ has slicing $\beta\cdot \cP$ defined by $(\beta \cdot \cP)(\phi) \coloneqq \beta(\cP(\phi))$ for each $\phi \in \bbR$.
One can check that this defines an action of $\bbB(\Gamma)$ on $\Stab(\cK_{\Gamma})$ by homeomorphisms.
\end{definition}
Let us first understand how this action translates through the forgetful map $\cZ\colon \Stab(\cK_\Gamma) \ra \Hom_{\bbZ}(\Lambda_{\Gamma}, \bbC)$.
Recall that the action of the Artin group $\bbB(\Gamma)$ on $\Lambda_{\Gamma}$ factors through the geometric representation of the Coxeter group $\bbW(\Gamma)$ on $\Lambda_{\Gamma}$; namely, with $\sigma_i$ and $s_i$ denoting the standard generators of $\bbB(\Gamma)$ and $\bbW(\Gamma)$ respectively,  the following diagram is commutative:
\[
\begin{tikzcd}
K_0(\cK) \ar[r, "q=-1"] \ar[d,"\sigma_i"] & \Lambda_\Gamma \ar[d,"s_i"]\\
K_0(\cK_{\Gamma}) \ar[r, "q=-1"] & \Lambda_\Gamma
\end{tikzcd}.
\]
The geometric representation of $\bbW(\Gamma)$ on $\Lambda_{\Gamma}$ induces an action on the $\bbC$-linear dual $\Lambda_{\Gamma}$, where the action on each $f\in \Hom_{\bbZ}(\Lambda_{\Gamma}, \bbC)$ is given by
\[
(s \cdot f)(x) \coloneqq f(s^{-1}(x)).
\]
This is known as the \emph{contragradient representation of} $\bbW(\Gamma)$.
With $\Hom_{\bbZ}(\Lambda_\Gamma, \bbC)$ equipped with the action of $\bbW(\Gamma)$ given by the contragradient representation, we see that the forgetful map $\cZ\colon \Stab(\cK_\Gamma) \ra \Hom_{\bbZ}(\Lambda_\Gamma, \bbC)$ is equivariant with respect to the action of $\bbB(\Gamma)$ and $\bbW(\Gamma)$, i.e.\ the following diagram is commutative:
\[
\begin{tikzcd}
\Stab(\cK_\Gamma) \ar[r, "\cZ"] \ar[d, "\sigma_i"] &  \Hom_{\bbZ}(\Lambda_\Gamma, \bbC) \ar[d,"s_i"]\\
\Stab(\cK_\Gamma) \ar[r, "\cZ"] &  \Hom_{\bbZ}(\Lambda_\Gamma, \bbC)
\end{tikzcd}.
\]
We recall that the space of hyperplane complement associated to $\bbW(\Gamma)$ can be defined using $\Hom_{\bbZ}(\Lambda_{\Gamma}, \bbC)$ as follows.
Let $\alpha_i := [P_i] \in \Lambda_\Gamma$ for each $i \in \Gamma_0$ be the (positive) simple roots of $\Lambda_\Gamma$.
The set of positive roots $\Phi^+$ is then given by
\[
\Phi^+ \coloneqq \bbW \cdot \{ \alpha_i \mid i \in \Gamma_0 \} \cap \sum_{i \in \Gamma_0} \bbZ_{\geq 0}\cdot [P_i].
\]
For each positive root $\alpha \in \Phi^+$ of $\bbW(\Gamma)$, define the complex hyperplane
\[
H_\alpha \coloneqq \{ Z \in \Hom_{\bbZ}(\Lambda_\Gamma,\bbC) \mid Z(\alpha) = 0 \},
\]
so that 
\[
V^c_{\Gamma} \coloneqq  \Hom_{\bbZ}(\Lambda_\Gamma, \bbC) \setminus \bigcup_{\alpha \in \Phi^+} H_{\alpha}
\]
is the \emph{hyperplane complement} associated to $\bbW(\Gamma)$.
The (contragradient) action of $\bbW(\Gamma)$ on the hyperplane complement is free and properly discontinuous and the fundamental group of the quotient space is isomorphic to the Artin group:
\[
\bbB(\Gamma) \cong \pi_1\left( V^c_{\Gamma}/\bbW(\Gamma) \right).
\]
Moreover, a fundamental domain of $V^c_\Gamma$ can be identified as follows.
\begin{proposition}[Complexified Weyl chamber] \label{prop:weylchamber}
A fundamental domain of $V^c_\Gamma$ with respect to the action of $\bbW(\Gamma)$ is given by
\[
C_{\Gamma}\coloneqq \{ Z \in \Hom_{\bbZ}(\Lambda_\Gamma, \bbC) \mid Z(P_i) \in \bbH \cup \bbR_{<0} \}.
\]
Moreover, every point in $\Hom_\bbZ(\Lambda_\Gamma, \bbC)$ lies in the $\bbW(\Gamma)$-orbit of the closure
\[
\overline{C_\Gamma} = \{ Z \in \Hom_{\bbZ}(\Lambda_\Gamma, \bbC) \mid Z(P_i) \in \bbH \cup \bbR \}.
\]
\end{proposition}
The following decomposition of $\Hom_\bbZ(\Lambda_\Gamma, \bbC)$ into (real-codimension one) walls and chambers will be helpful.
For each positive root $\alpha \in \Phi^+$, the (type II) wall associated to $\alpha$, denoted by $\frW_\alpha$, is defined as 
\[
\frW_\alpha := \{ Z \in \Hom_{\bbZ}(\Lambda_\Gamma, \bbC) \mid Z(\alpha) \in \R\}.
\]
Note that these walls are \emph{real-codimension one} and are \emph{not} the same as the complex-codimension one hyperplanes $H_\alpha$; each hyperplane $H_\alpha$ is just a subset of $\frW_\alpha$.
Removing these walls from $\Hom_\bbZ(\Lambda_\Gamma, \bbC)$ results in a disjoint union of open chambers 
\begin{equation}\label{eq:walldecomp}
\Hom_\bbZ(\Lambda_\Gamma, \bbC) \setminus \bigcup_{\alpha \in \Phi^+} \frW_\alpha = \coprod_{w \in \bbW(\Gamma)} w \cdot C_\gamma^o, 
\end{equation}
where $C_\Gamma^o$ is the interior of $C_\Gamma$:
\[
C_\Gamma^o = \{ Z \in \Hom_{\bbZ}(\Lambda_\Gamma, \bbC) \mid Z(P_i) \in \bbH\}.
\]
Indeed, the boundary points of $C_\Gamma^o$ all live on some wall $\frW_{\alpha_i}$ and $w \cdot \frW_{\alpha} = \frW_{w(\alpha)}$.
The equality \eqref{eq:walldecomp} follows from \cref{prop:weylchamber}.

Going back to our running example $\Gamma = A_2$, we thus have 6 open chambers -- one for each element of the Coxeter group $\bbW(A_2) \cong S_3$ -- that are also the connected components of $\Hom_\bbZ(\Lambda_{A_2}, \bbC)$ after the removal of the 3 real-codimension one walls:
\begin{align*}
\frW_{\alpha_1} &= \{Z \in \Hom_\bbZ(\Lambda_{A_2}, \bbC) \mid Z(\alpha_1) = Z(P_1) \in \R\}; \\
\frW_{\alpha_2} &= \{Z \in \Hom_\bbZ(\Lambda_{A_2}, \bbC) \mid Z(\alpha_2) = Z(P_2) \in \R\}; \text{ and} \\
\frW_{\alpha_1 + \alpha_2} &= \{Z \in \Hom_\bbZ(\Lambda_{A_2}, \bbC) \mid Z(\alpha_1 + \alpha_2) = Z(P_1)+Z(P_2) \in \R\}.
\end{align*}
\begin{remark}
The preimage of $\cZ$ of the walls described above are real-codimension one submanifolds in $\Stab(\cK_\Gamma)$ known as type II walls.
Crossing type II walls in the stability manifold indicates a change in the (standard) heart $\cP(0,1]$ associated to stability conditions $(\cP, Z)$.
There are also type I walls which describe changes to the stability of objects, which we will not discuss here.
\end{remark}

Now that we understand the local picture $\Hom_{\bbZ}(\Lambda_\Gamma, \bbC)$, let us return to the global picture $\Stab(\cK_\Gamma)$.
In general, $\Stab(\cD)$ need not be connected, so we shall just consider the connected component\footnote{It is known that for $\Gamma$ of ADE type, $\Stab(\cK_{\Gamma})$ is connected \cite{AMY_silting, BDL_2CY}, but we will not need this result.} $\Stab^*(\cK_{\Gamma}) \subseteq \Stab(\cK_{\Gamma})$ that contains the stability conditions coming from stability functions on the standard heart $\cH_{lin}$.
Abusing notation, we shall let $\Stab(\cH_{lin}) \subset \Stab^*(\cK_{\Gamma})$ denote the subset of stability conditions defined by stability functions on the standard heart $\cH_{lin}$.
Any stability function on the standard heart $\cH_{lin}$ is completely (and uniquely) defined by $Z(P_i) \in \bbH \cup \bbR_{<0}$ for each $i \in \Gamma_0$ (see \cref{eg:A2stab} for the example $\Gamma = A_2$).
In particular, we have that
\begin{equation}\label{eq:heartupperplane}
    \cZ\colon \Stab(\cH_{lin}) \xra{\cong} C_{\Gamma}.
\end{equation}
Our first step is to prove the following lifted version of \cref{prop:weylchamber}.
\begin{proposition}\label{prop:stabfundamentaldomain}
For any stability condition $\varsigma \in \Stab^*(\cK_\Gamma)$, there exists a unique $\beta \in \bbB(\Gamma)$ such that $\beta\cdot \varsigma \in \Stab(\cH_{lin})$. In particular, the action of $\bbB(\Gamma)$ on $\Stab^*(\cK_{\Gamma})$ is free and $\Stab(\cH_{lin})$ is a fundamental domain of $\Stab^*(\cK_{\Gamma})$.
\end{proposition}

We saw in \cref{eg:deformP1} and \cref{ex:crossnegative} that the stability conditions coming from the hearts $\sigma_1^{\pm 1}(\cH_{lin})$ and $\cH_{lin}$ are ``close by''; the following is a general version of that result:
\begin{lemma}[\protect{\cite[Lemma 3.5]{bridgeland2009kleinian}}]\label{lem:neighbourhearts}
Suppose $(\cP,Z) \in \Stab(\cK_\Gamma)$ is defined by a stability function on the standard heart $\cH_{lin}$ satisfying $Z(P_i) \in \bbR_{<0}$ and $Z(P_j) \in \bbH$ for all $j \neq i$.
Then there exists an open neighbourhood $U \subset \Stab(\cK_\Gamma)$ of $(\cP,Z)$ such that the stability conditions in $U$ are either given by stability functions on $\cH_{lin}$ or $\sigma_i(\cH_{lin})$.
Conversely, suppose $(\cP,Z)$ is defined by a stability function on $\sigma_i^{-1}\cH_{lin}$ satisfying $Z(P_i[1]) \in \bbR_{<0}$ (so $Z(P_i) \in \bbR_{>0}$) and $Z(P_j) \in \bbH$ for all $j \neq i$.
Then there exists an open neighbourhood $U \subset \Stab(\cK_\Gamma)$ of $(\cP,Z)$ such that the stability conditions in $U$ are either given by stability functions on $\sigma_i^{-1}(\cH_{lin})$ or $\cH_{lin}$.
\end{lemma}
Note that the lemma above also implies that the action of $\bbB(\Gamma)$ on $\Stab(\cK_{\Gamma})$ preserves the connected component $\Stab^*(\cK_\Gamma)$.
\begin{proof}[Proof of \cref{prop:stabfundamentaldomain}]
Let $\xi \in \Stab^*(\cK_\Gamma)$.
Pick a path $p$ connecting some $\varsigma \in \Stab(\cH_{lin})$ to $\xi$ (connected implies path connected in a manifold).
The projection $\cZ(p)$ of the path $p$ is then a path in $\Hom_\bbZ(\Lambda_\Gamma,\bbC)$.
Using the local homeomorphism property of $\cZ$, we may deform $p$ slightly so that $\cZ(p)$ only passes through finitely many walls in $\Hom_\bbZ(\Lambda_\Gamma,\bbC)$ (no infinite ``wiggling'' around a wall'') and $\cZ(p)$ only passes through real-codimension one walls (in the $A_2$ case, this means the path $p$ only passes through walls with at most one of $Z(P_1), Z(P_2)$ and $Z(P_1)+Z(P_2)$ having zero imaginary part).
As such, every time $p$ passes through a wall of $\Stab(\cH_{lin}) \cong C_\Gamma$, \cref{lem:neighbourhearts} says that either $\sigma_i$ or $\sigma_i^{-1}$ can be applied so that we are back at stability conditions in $\Stab(\cH_{lin})$.
Given the condition on path $p$ this can only happen finitely many times, and so $\beta\cdot \xi \in \Stab(\cH_{lin})$ for some $\beta \in \bbB(\Gamma)$ as required.

To prove uniqueness, it is now sufficient to show that for all $\varsigma \in \Stab(\cH_{lin})$, $\beta \cdot \varsigma = \varsigma$ implies $\beta = \id$.
Note that $\beta \cdot \varsigma = \varsigma$ implies in particular that $\beta$ preserves $\cH_{lin}$. 
One can explicitly show that any composition of the endofunctors $\sigma_s$ that preserves the heart $\cH_{lin}$ is isomorphic to the identity functor (this can be done via realising the endofunctors $\sigma_i$ as actual complexes of graded bimodules, or see e.g.\ \cite[Remark 3.7]{bridgeland2009kleinian}).
Moreover, for $\Gamma$ of ADE type, the assignment of elements in $\bbB(\Gamma)$ to endofunctors of $\cK_\Gamma$ given in \cref{thm:braidaction} is known to be faithful (see references after \cref{thm:braidaction}), hence the result follows.
\end{proof}

We can now deduce that the image of $\cZ$ is the hyperplane complement $V^c_\Gamma$.
\begin{proposition}
The image of $\cZ\colon\Stab^*(\cK_\Gamma) \ra \Hom_\bbZ(\Lambda_\Gamma, \bbC)$ is equal to $V^c_\Gamma$.
\end{proposition}
\begin{proof}
Recall from \cref{prop:stabfundamentaldomain} and \cref{prop:weylchamber} that $\Stab(\cH_{lin})$ and $C_{\Gamma}$ are fundamental domains of $\Stab^*(\cK_\Gamma)$ and $V^c_\Gamma$ respectively.
Moreover, $\cZ$ identifies the two fundamental domains.
The result now follows from the fact that $\cZ$ is equivariant with respect to $\bbB(\Gamma)$ and $\bbW(\Gamma)$.
\end{proof}

Moreover, this allows us to show that the action of $\bbB(\Gamma)$ on $\Stab^*(\cK_\Gamma)$ is properly discontinuous:
\begin{proposition}
The action of $\bbB(\Gamma)$ on $\Stab^*(\cK_{\Gamma})$ is properly discontinuous.
\end{proposition}
\begin{proof}
On the hyperplane complement $V^c_\Gamma$, the action of $\bbW(\Gamma)$ is properly discontinuous.
As such, for each $(\cP,Z) \in \Stab^*(\cK_{\Gamma})$ we can pick a small enough neighbourhood $U_{(\cP,Z)}$ such that $\cZ$ restricts to a homeomorphism onto an open neighbourhood $U_Z$ of $Z = \cZ(\cP,Z)$ satisfying 
\begin{equation}\label{eq:properlydiscont}
    w\cdot U_Z \cap U_Z \neq \emptyset \implies w = \id \in \bbW(\Gamma).
\end{equation}
Suppose $\beta \cdot U_{(\cP,Z)} \cap U_{(\cP,Z)} \neq \emptyset$.
Applying $\cZ$ and using its equivariant property, \eqref{eq:properlydiscont} implies that the image of $\beta$ in $\bbW(\Gamma)$ must be the identity element.
As such, the action of $\beta$ on any stability condition will leave its stability function invariant.
Now let $(\cP',Z') \in U_{(\cP,Z)}$ such that $\beta\cdot (\cP',Z') = (\beta\cdot \cP', Z') \in U_{(\cP,Z)}$.
Since $U_{(\cP,Z)}$ could be made arbitrarily small, we could assume without lost of generality that all stability conditions in it has slicing distance from $\cP$ less than $1/2$, so that we have 
\[
d(\cP', \cP)<1/2, \text{ and} \quad (\beta\cdot \cP', \cP)<1/2.
\]
Thus, we obtain 
\[
d(\cP', \beta\cdot \cP') \leq d(\cP', \cP) + d(\cP, \beta\cdot \cP') < 1, 
\]
where \cref{lem:uniqueslicing<1} implies that $\beta \cdot \cP' = \cP'$.
In particular, we have found $\beta \in \bbB(\Gamma)$ such that $\beta\cdot (\cP', Z') = (\cP', Z')$, where the uniqueness from \cref{prop:stabfundamentaldomain} implies that $\beta = \id$ as required.
\end{proof}

\begin{theorem}\label{thm:covering}
The composition 
\[
\Stab^*(\cK_\Gamma) \xra{\cZ} V^c_\Gamma \twoheadrightarrow V^c_\Gamma / \bbW(\Gamma)
\]
is a covering map, with $\bbB(\Gamma)$ acting via deck transformations.
\end{theorem}
\begin{proof}
Since $\bbB(\Gamma)$ acts on $\Stab^*(\cK_\Gamma)$ freely and properly discontinuously, we know that the quotient
\[
\Stab^*(\cK_\Gamma) \twoheadrightarrow \Stab^*(\cK_\Gamma)/\bbB(\Gamma)
\]
is a covering map.
Recall that $\Stab(\cH_{lin})$ and $C_\Gamma$ are fundamental domains of $\Stab^*(\cK_\Gamma)$ and $V^c_\Gamma$ respectively.
Since $\cZ$ is equivariant with respect to the actions of $\bbB(\Gamma)$ and $\bbW(\Gamma)$, \eqref{eq:heartupperplane} shows that the quotient spaces $\Stab^*(\cK_\Gamma)/\bbB(\Gamma)$ and $V^c_\Gamma / \bbW(\Gamma)$ are identified via $\cZ$.
In particular, the composition in the theorem is indeed a covering map.
The equivariant property of $\cZ$ also implies that $\bbB(\Gamma)$ acts by deck transformations.
\end{proof}

\begin{corollary}
The $K(\pi,1)$ conjecture holds of $\Gamma$ of ADE type.
\end{corollary}
\begin{proof}
For $\Gamma$ of ADE type, the stability manifold $\Stab^*(\cK_\Gamma)$ is known to be contractible \cite[Theorem A]{QW_18}. As such, the covering in \cref{thm:covering} is actually a universal covering of the hyperplane complement.
Being contractible, it also proves the $K(\pi,1)$ conjecture for $\bbB(\Gamma)$.
\end{proof}
\begin{remark}
It is worth mentioning here that there are different triangulated categories (other than $\cK_\Gamma$) whose space of stability conditions can be used to study the $K(\pi,1)$ conjecture.
The triangulated category $\cK_\Gamma$ presented here is the most ``combinatorial'' one (in the author's very biased opinion), and other popular choices of triangulated categories include (some of which are related to ours via (dg)Koszul duality):
\begin{itemize}
    \item Preprojective algebras \cite{bridgeland2009kleinian,ikeda2014stability};
    \item Ginzburg (dg) algebras \cite{QW_18};
    \item Contraction algebras \cite{AW_22}.
\end{itemize}
The approach using contraction algebras \cite{AW_22} also allowed August--Wemyss to reprove the $K(\pi,1)$ conjecture for a few extra non-simply-laced cases: $I_2(m)$ for $m = 4,5,6$ and 8.
\end{remark}
\begin{remark}
For $\Gamma$ outside of ADE types (but simply-laced), Ikeda showed\footnote{Technically, there is a gap in Ikeda's argument, but it is fixable.} \cite{ikeda2014stability} that $\Stab(\cK_\Gamma)$ is a covering space of $S^1 \times (V^c_\Gamma/\bbW(\Gamma))$, where contractibility of $\Stab(\cK_\Gamma)$ would also prove the $K(\pi,1)$ conjecture.
For the affine case, see also \cite{bridgeland2009kleinian}.
\end{remark}
\begin{remark} \label{rem:nonsimplylaced}
The case of non-simply-laced $\Gamma$ can be approached via cutting out a nice (closed) submanifold of $\Stab^*(\Gamma)$ using fusion-equivariant stability conditions \cite{Heng_PhDthesis, DHL_fusionstab}.
See \cite[Theorem 5.5]{QZ_fusionstab} where they treat the case of (non-simply-laced) finite-type Coxeter diagrams.
\end{remark}

\appendix
\section{Refinement of filtrations and triangulations}\label{app:refinement}
The main purpose of this appendix is to introduce the octahedral axiom of a triangulated category via a dual presentation of exact triangles (where they really are triangles!), as introduced in \cite{dyckerhoff_kapranov_2018} -- the upshot is that under this dual presentation, the octahedral axiom corresponds to flips of triangulations.
Throughout we fix $\cD$ to be a triangulated category.

Let us start by introducing the dual presentation of exact triangles: 
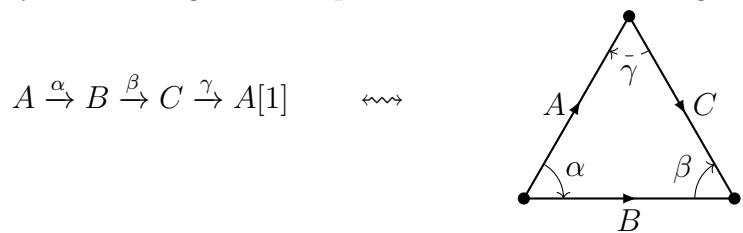
\begin{center}
\begin{tikzpicture}[scale=1.4]
\tikzset{
  arrow/.pic={\path[tips,every arrow/.try,->,>=#1] (0.1,0) -- +(.1pt,0);},
  pics/arrow/.default={triangle 90}
}

\coordinate (X) at (0,0);
\coordinate (Y) at (2,0);
\coordinate (Z) at (1,1.732);

\begin{scope}[very thick,nodes={sloped,allow upside down}]
\draw[thick] (X) -- pic{arrow=latex} (Y) ;
\draw[thick] (X) -- pic{arrow=latex} (Z) ;
\draw[thick] (Z) -- pic{arrow=latex} (Y) ;
\end{scope}

\pic [draw, <-, swap, "$\alpha$", angle radius = 15, angle eccentricity=1.5] {angle=Y--X--Z};
\pic [draw, <-, swap, "$\beta$", angle radius = 15, angle eccentricity=1.5] {angle=Z--Y--X};
\pic [draw, <-, swap, dashed,"$\gamma$", angle radius = 15, angle eccentricity=1.5] {angle=X--Z--Y};

\filldraw[color=black!, fill=black!]  (0,0) circle [radius=1.5pt];
\filldraw[color=black!, fill=black!]  (2,0) circle [radius=1.5pt];
\filldraw[color=black!, fill=black!]  (1,1.732) circle [radius=1.5pt];

\node[left] at (0.5, 0.9) {$A$};
\node[below] at (1, 0) {$B$};
\node[right] at (1.5, 0.9) {$C$};

\node at (-3,1) {$A \xra{\alpha} B \xra{\beta} C \xra{\gamma} A[1] \qquad \leftrightsquigarrow$};
\end{tikzpicture}
\end{center}

Using the dual presentation, a filtration of $X$ by exact triangles:
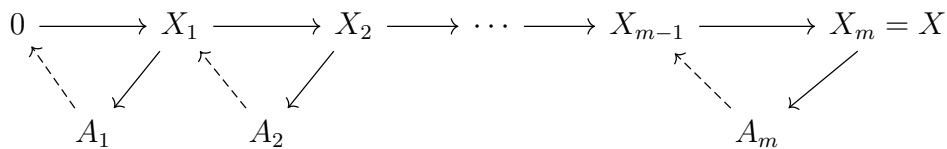
\[
\begin{tikzcd}[column sep = 3mm]
0
	\ar[rr] &
{}
	{} &
X_1
	\ar[rr] \ar[dl]&
{}
	{} &	
X_2
	\ar[rr] \ar[dl]&
{}
	{} &	
\cdots
	\ar[rr] &
{}
	{} &
X_{m-1}
	\ar[rr] &
{}
	{} &
X_m = X
	\ar[dl] \\
{}
	{} &
A_1
	\ar[lu, dashed] &
{}
	{} &
A_2
	\ar[lu, dashed] &
{}
	{} &
{}
	{} &
{}
	{} &
{}
	{} &
{}
	{} &
A_m
	\ar[lu, dashed] &
\end{tikzcd}
\]
is presented as the following polygon (direction of morphisms implied by the orientation of the edges)
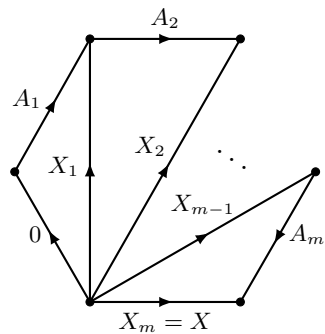
\begin{center}
\begin{tikzpicture}
\tikzset{
  arrow/.pic={\path[tips,every arrow/.try,->,>=#1] (0.1,0) -- +(.1pt,0);},
  pics/arrow/.default={triangle 90}
}

\coordinate (A) at (0,0);
\coordinate (B) at (2,0);
\coordinate (C) at (-1,1.732);
\coordinate (D) at (3,1.732);
\coordinate (E) at (0,3.5);
\coordinate (F) at (2,3.5);

\begin{scope}[very thick,nodes={sloped,allow upside down}]
\draw[thick] (A) -- pic{arrow=latex} (B) ;
\draw[thick] (A) -- pic{arrow=latex} (C) ;
\draw[thick] (A) -- pic{arrow=latex} (D) ;
\draw[thick] (A) -- pic{arrow=latex} (E) ;
\draw[thick] (A) -- pic{arrow=latex} (F) ;
\draw[thick] (D) -- pic{arrow=latex} (B) ;
\draw[thick] (C) -- pic{arrow=latex} (E) ;
\draw[thick] (E) -- pic{arrow=latex} (F) ;
\end{scope}

\filldraw[color=black!, fill=black!]  (A) circle [radius=1.5pt];
\filldraw[color=black!, fill=black!]  (B) circle [radius=1.5pt];
\filldraw[color=black!, fill=black!]  (C) circle [radius=1.5pt];
\filldraw[color=black!, fill=black!]  (D) circle [radius=1.5pt];
\filldraw[color=black!, fill=black!]  (E) circle [radius=1.5pt];
\filldraw[color=black!, fill=black!]  (F) circle [radius=1.5pt];

\node[left] at (-0.5, 0.9) {\scriptsize $0$};
\node[left] at (0, 1.8) {\scriptsize $X_1$};
\node[left] at (-0.5, 2.7) {\scriptsize $A_1$};
\node[above] at (0.8, 1.8) {\scriptsize $X_2$};
\node[above] at (1, 3.5) {\scriptsize $A_2$};
\node[above] at (1.5, 1) {\scriptsize $X_{m-1}$};
\node[below] at (1, 0) {\scriptsize $X_m=X$};
\node[right] at (2.5, 0.9) {\scriptsize $A_m$};
\node[right] at (1.5, 2) {$\ddots$};

\end{tikzpicture}
\end{center}
For our purposes, the triangulation of a polygon as above is called the \emph{standard triangulation}.
Note that the left most triangle can be ``absorbed'' into the edge $X_1$, since the morphism between $X_1$ and $A_1$ must be an isomorphism (see \cref{ex:exacttriangle}).

Now suppose each $A_i$ has a filtration by some $E^i_j$'s.
As above, thinking of the filtrations as polygons with the standard triangulation (with outer edges labelled by $E^i_j$'s), we can attach each of these polygons to the one above to obtain a single larger polygon.
Up to re-triangulation, this larger polygon would have the standard triangulation -- hence a filtration.

This is where the \emph{octahedral axiom} comes into play.
A ``standard presentation'' of the octahedral axiom will look as follows.
Suppose we have two morphisms $A_1 \xra{f} A_2$ and $A_2 \xra{g} A_3$ in $\cD$.
Then we can consider exact triangles associated to $f, g$ and the composition $gf$ ($A_{12}$, $A_{23}$ and $A_{13}$ are cones of $f$, $g$ and $gf$ respectively):
\[
A_1 \xra{f} A_2 \ra A_{12} \ra A_1[1], \quad A_2 \xra{g} A_3 \ra A_{23} \ra A_2[1], \quad A_1 \xra{gf} A_3 \ra A_{13} \ra A_1[1].
\]
Then there exists (not necessarily unique) morphisms $\psi_1, \psi_2, \psi_3$ in $\cD$ relating the three exact triangles as follows, where each part of the diagram is commutative:
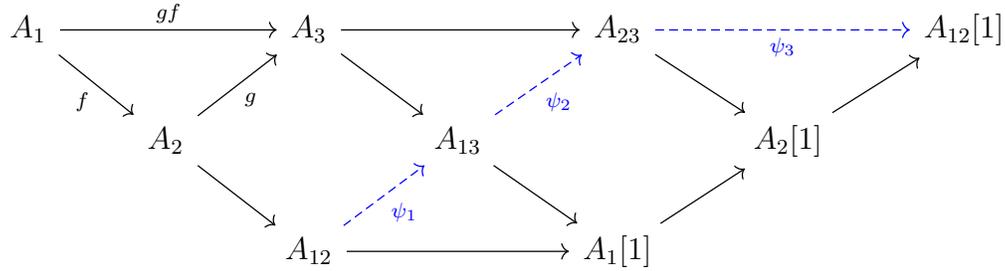
\[\begin{tikzcd}
	{A_1} && {A_3} && {A_{23}} && {A_{12}[1]} \\
	& {A_2} && {A_{13}} && {A_2[1]} \\
	&& {A_{12}} && {A_1[1]}
	\arrow["f"', from=1-1, to=2-2]
	\arrow["g"', from=2-2, to=1-3]
	\arrow["gf", from=1-1, to=1-3]
	\arrow[from=1-3, to=1-5]
	\arrow[blue, "{\psi_3}"', dashed, from=1-5, to=1-7]
	\arrow[from=2-2, to=3-3]
	\arrow[from=1-3, to=2-4]
	\arrow[blue, "{\psi_1}"', dashed, from=3-3, to=2-4]
	\arrow[blue, "{\psi_2}"', dashed, from=2-4, to=1-5]
	\arrow[from=2-4, to=3-5]
	\arrow[from=1-5, to=2-6]
	\arrow[from=3-3, to=3-5]
	\arrow[from=3-5, to=2-6]
	\arrow[from=2-6, to=1-7]
\end{tikzcd}\]
In the dual presentation, the octahedral axiom implies that we can perform a triangulation flip on a square, switching from one triangulation to the other:
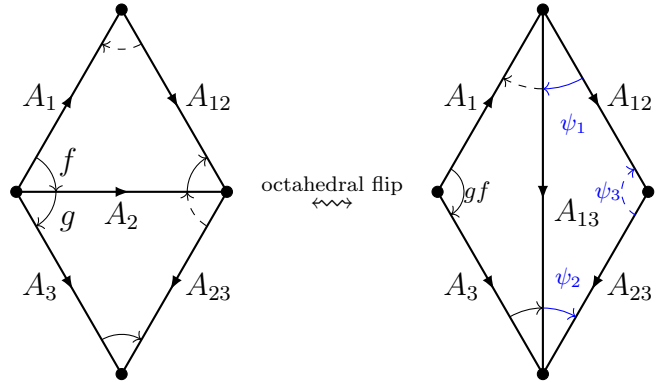
\begin{center}
\begin{tikzpicture}[scale=1.4]
\tikzset{
  arrow/.pic={\path[tips,every arrow/.try,->,>=#1] (0.1,0) -- +(.1pt,0);},
  pics/arrow/.default={triangle 90}
}

\coordinate (X) at (0,0);
\coordinate (Y) at (2,0);
\coordinate (Z) at (1,1.732);
\coordinate (W) at (1,-1.732);

\begin{scope}[very thick,nodes={sloped,allow upside down}]
\draw[thick] (X) -- pic{arrow=latex} (Y) ;
\draw[thick] (X) -- pic{arrow=latex} (Z) ;
\draw[thick] (Z) -- pic{arrow=latex} (Y) ;
\draw[thick] (X) -- pic{arrow=latex} (W) ;
\draw[thick] (Y) -- pic{arrow=latex} (W) ;
\end{scope}

\pic [draw, <-, swap, angle radius = 15, angle eccentricity=1.5, "$f$"] {angle=Y--X--Z};
\pic [draw, <-, swap, angle radius = 15, angle eccentricity=1.5] {angle=Z--Y--X};
\pic [draw, <-, swap, dashed, angle radius = 15, angle eccentricity=1.5] {angle=X--Z--Y};
\pic [draw, <-, swap, angle radius = 15, angle eccentricity=1.5, "$g$"] {angle=W--X--Y};
\pic [draw, <-, swap, angle radius = 15, angle eccentricity=1.5] {angle=Y--W--X};
\pic [draw, <-, swap, dashed, angle radius = 15, angle eccentricity=1.5] {angle=X--Y--W};

\filldraw[color=black!, fill=black!]  (0,0) circle [radius=1.5pt];
\filldraw[color=black!, fill=black!]  (2,0) circle [radius=1.5pt];
\filldraw[color=black!, fill=black!]  (1,1.732) circle [radius=1.5pt];
\filldraw[color=black!, fill=black!]  (1,-1.732) circle [radius=1.5pt];

\node[left] at (0.5, 0.9) {$A_1$};
\node[below] at (1, 0) {$A_2$};
\node[right] at (1.5, 0.9) {$A_{12}$};
\node[left] at (0.5, -0.9) {$A_3$};
\node[right] at (1.5, -0.9) {$A_{23}$};

\node at (3,0) {$\overset{\text{octahedral flip}}{\leftrightsquigarrow}$};

\coordinate (X) at (4,0);
\coordinate (Y) at (6,0);
\coordinate (Z) at (5,1.732);
\coordinate (W) at (5,-1.732);

\begin{scope}[very thick,nodes={sloped,allow upside down}]
\draw[thick] (Z) -- pic{arrow=latex} (W) ;
\draw[thick] (X) -- pic{arrow=latex} (Z) ;
\draw[thick] (Z) -- pic{arrow=latex} (Y) ;
\draw[thick] (X) -- pic{arrow=latex} (W) ;
\draw[thick] (Y) -- pic{arrow=latex} (W) ;
\end{scope}

\pic [draw, <-, swap, angle radius = 10, angle eccentricity=1.5, "\scriptsize$gf$"] {angle=W--X--Z};
\pic [draw, <-, swap, angle radius = 30, angle eccentricity=1.5, "\scriptsize$\psi_1$", blue] {angle=W--Z--Y};
\pic [draw, <-, swap, dashed, angle radius = 30, angle eccentricity=1.5] {angle=X--Z--W};
\pic [draw, <-, swap, angle radius = 25, angle eccentricity=1.5] {angle=Z--W--X};
\pic [draw, <-, swap, angle radius = 25, angle eccentricity=1.5, "\scriptsize$\psi_2$", blue] {angle=Y--W--Z};
\pic [draw, <-, swap, dashed, angle radius = 10, angle eccentricity=1.5, "\scriptsize$\psi_3$", blue] {angle=Z--Y--W};

\filldraw[color=black!, fill=black!]  (X) circle [radius=1.5pt];
\filldraw[color=black!, fill=black!]  (Y) circle [radius=1.5pt];
\filldraw[color=black!, fill=black!]  (Z) circle [radius=1.5pt];
\filldraw[color=black!, fill=black!]  (W) circle [radius=1.5pt];

\node[left] at (4.5, 0.9) {$A_1$};
\node[right] at (5, -0.2) {$A_{13}$};
\node[right] at (5.5, 0.9) {$A_{12}$};
\node[left] at (4.5, -0.9) {$A_3$};
\node[right] at (5.5, -0.9) {$A_{23}$};
\end{tikzpicture}
\end{center}
Since every triangulation can be achieved through a finite sequence of flips, we can always achieve the standard triangulation, as required.

\printbibliography

\end{document}